\documentclass[a4paper]{article}
\usepackage{amsmath, amsfonts}
\usepackage{amssymb, latexsym}
\usepackage{amsthm}
\usepackage{mathtools}
\usepackage{tikz}
\makeatletter
\newcommand{\leqnomode}{\tagsleft@true\let\veqno\@@leqno}
\newcommand{\reqnomode}{\tagsleft@false\let\veqno\@@eqno}
\makeatother

\makeatletter
\def\widebreve{\mathpalette\wide@breve}
\def\wide@breve#1#2{\sbox\z@{$#1#2$}%
     \mathop{\vbox{\m@th\ialign{##\crcr
\kern0.08em\brevefill#1{0.8\wd\z@}\crcr\noalign{\nointerlineskip}%
                    $\hss#1#2\hss$\crcr}}}\limits}
\def\brevefill#1#2{$\m@th\sbox\tw@{$#1($}%
  \hss\resizebox{#2}{\wd\tw@}{\rotatebox[origin=c]{90}{\upshape(}}\hss$}
\makeatletter

\usepackage[%dvipdfmx,
colorlinks=true,       % false: boxed links; true: colored links
linkcolor=blue,          % color of internal links (change box color with linkbordercolor)
citecolor=blue,        % color of links to bibliography
plainpages=false,      % do page number anchors as plain Arabic
pdfpagelabels,
]{hyperref}

\usepackage[shortlabels]{enumitem}
\setlist[enumerate]{leftmargin=*,align=left,labelindent=\parindent}

\newcommand{\Bin}{\left\{ 0,1 \right\}}
\newcommand{\Ter}{\left\{ 0,1, 2\right\}}

\DeclareMathOperator{\imp}{\,\rightarrow\,}
\DeclareMathOperator{\defeqiv}{\stackrel{\textup{def}}{\iff}}
\DeclareMathOperator{\defeql}{\stackrel{\textup{def}}{\  =\  }}
\newcommand{\dotminus}{\mathbin{\ooalign{\hss\raise.5ex\hbox{$\cdot$}\hss\crcr$-$}}}

\newcommand{\Nat}{\mathbb{N}}

\newcommand{\FSeq}{\Nat^{*}}
\newcommand{\BSeq}{\Bin^{*}}
\newcommand{\TSeq}{\Ter^{*}}

\newcommand{\Baire}{\Nat^{\Nat}}
\newcommand{\Cantor}{\Bin^{\Nat}}
\newcommand{\CSet}{\mathbb{C}}
\newcommand{\TTree}{\Ter^{\Nat}}
\newcommand{\UInt}{[0,1]}
\newcommand{\RInt}{\mathbb{I}}
\newcommand{\Real}{\mathbb{R}}
\newcommand{\LoebReal}{\mathbb{R}_{\mathrm{S}}}
\newcommand{\Reg}{\lh{\UInt}}
\newcommand{\Rat}{\mathbb{Q}}
\newcommand{\Int}{\mathbb{Z}}
\newcommand{\RatInt}{\mathbb{T}}

\newcommand{\nil}{\langle\, \rangle}
\newcommand{\seq}[1]{\langle#1\rangle}

\newcommand{\lh}[1]{\lvert #1 \rvert}
\newcommand{\fst}[1]{{#1}'}
\newcommand{\snd}[1]{{#1}''}

\newcommand{\qfAC}{\mathrm{QF\text{-}\mathrm{AC}}}

\newcommand{\CFAN}{\textrm{\textup{CFT}}}
\newcommand{\DFAN}{\mathrm{DFT}}

\newcommand{\UCc}{\mathrm{UC_{c}}}
\newcommand{\UCTcC}{\mathrm{UCT_{c}}_{\BSeq}}
\newcommand{\UCTc}{\mathrm{UCT_{c}}}
\newcommand{\UCTcT}{\mathrm{UCT_{c}'}}

\newcommand{\UCT}{\mathrm{UCT}}

\DeclareMathOperator{\id}{\mathrm{id}}

\newcommand{\HAw}{\mathrm{HA}^{\omega}}

\newtheorem{theorem}{Theorem}[section]
\newtheorem{proposition}[theorem]{Proposition}
\newtheorem{lemma}[theorem]{Lemma}
\newtheorem{corollary}[theorem]{Corollary}

\theoremstyle{definition}
\newtheorem{definition}[theorem]{Definition}
\theoremstyle{remark}
\newtheorem{remark}[theorem]{Remark}

\newtheorem{notation}[theorem]{Notation}
 \numberwithin{equation}{section}

\title{Decidable fan theorem and uniform continuity theorem with continuous moduli}

\usepackage{authblk}

\author{Makoto Fujiwara}
\affil{School of Science and Technology, Meiji University\authorcr
1-1-1 Higashi-Mita, Tama-ku, Kawasaki-shi, Kanagawa 214-8571, Japan\authorcr
\texttt{makotofujiwara@meiji.ac.jp}}

\author{Tatsuji Kawai}
\affil{Japan Advanced Institute of Science and Technology\authorcr
1-1 Asahidai, Nomi, Ishikawa 923-1292, Japan\authorcr
\texttt{tatsuji.kawai@jaist.ac.jp}}

\date{}

\begin{document}
\maketitle
\begin{abstract}
The uniform continuity theorem $(\UCT)$ states that every pointwise
continuous real-valued function on the unit interval is uniformly
continuous. In constructive mathematics, $\UCT$ is stronger than the
decidable fan theorem $(\DFAN)$; however, Loeb [Ann.\ Pure Appl.\
Logic, 132(1):51--66, 2005] has shown that the two principles become
equivalent with a suitable coding of ``continuous functions'' as
type-one objects. The question remains whether $\DFAN$
can be characterised by a weaker version of $\UCT$ using a natural
subclass of pointwise continuous functions without such a coding.
We show that when ``pointwise continuous'' is replaced with
``having a continuous modulus'', $\UCT$  becomes equivalent to
$\DFAN$. We also show that this weakening of $\UCT$ is equivalent to a
similar principle for real-valued functions on the Cantor space $\Cantor$.
These results extend Berger's characterisation of $\DFAN$ by
the similar principle for functions from $\Cantor$ to $\Nat$, and
unifies these characterisations of $\DFAN$ in terms of functions
having continuous moduli.
Furthermore, we directly show that the continuous real-valued functions on the
unit interval having  continuous moduli are exactly those functions which
admit the coding of ``continuous functions'' due to Loeb.
Our result allows us to interpret her work in the usual
context of mathematics.
\medskip

\noindent\textsl{Keywords:}
Constructive reverse mathematics;
Uniform continuity theorem;
Real-valued function;
Continuous modulus;
Fan theorem \\[3pt]
\noindent\textsl{MSC2010:}
03F60; % Constructive and recursive analysis
26E40; % Constructive real analysis
03F55; % Intuitionistic mathematics
03B30  % Reverse math
% 03F35; % Second- and higher-order arithmetic and fragments
% 03F50 % Metamathematics of constructive systems
\end{abstract}
%\listoftodos
\section{Introduction}\label{sec:Introduction}

In 1927~\cite{BrouwerDomainsofFunctions}, Brouwer showed that every
real-valued function on the unit interval is uniformly continuous. As Brouwer emphasised, the crucial role was played by the intuitionistic principle called the \emph{fan theorem}, which has become a subject of intensive study in constructive reverse mathematics~\cite{ConstRevMatheCompactness}.\footnote{One may also
notice that the version of fan theorem used in his proof is the
decidable fan theorem.}

The focus of this paper is on the relation between the fan theorem and
 the uniform continuity of
real-valued functions. Brouwer's work hints that there is a  strong connection between
the two, but their precise relation remains somewhat subtle.
In this context, the most natural statement to look at is the
\emph{uniform continuity theorem}:
\begin{description}
 \item[($\UCT$)] 
  Every pointwise continuous function $f \colon \UInt \to \Real$ is uniformly
  continuous. 
\end{description}
Bridges and Diener~\cite{PseudoCompUIntEquivUCT} gave various analytic
statements equivalent to $\UCT$.\footnote{The uniform continuity
  theorem is constructively equivalent to
  an analogous statement where the domain of $f$ is replaced with
$\Cantor$ or an arbitrary compact metric space; see Bridges and
Diener~\cite[Theorem~10]{PseudoCompUIntEquivUCT}.}
However, the logical relation between $\UCT$
and the fan theorem is unsettled.
The principle $\UCT$ is weaker than the fan theorem
for monotone $\Pi^{0}_{1}$ bars and stronger than the continuous fan
theorem ($\CFAN$)~\cite{BergerUCandcFT}, but no fan theoretic
characterisation of $\UCT$ has been known (see Diener and
Loeb~\cite{DienerLoebSeqRealValonUInt} for a concise summary of
equivalents of various forms of fan theorem).

On the positive side, Loeb \cite{Loeb05} introduced a coding of
continuous functions and uniformly continuous functions from $\UInt$ to $\Real$ in the context of intuitionistic
second order arithmetic, and showed that $\UCT$ is equivalent to the
decidable fan theorem ($\DFAN$) with respect to the coding of
continuous functions.  In order to define continuous functions from
$\UInt$ to $\Real$ in the second order arithmetic where the type two
functionals are not available, she encodes a continuous function as a
type one function with certain properties. However, the encoding makes
it difficult to compare her version of $\UCT$ with the standard
version of $\UCT$. In particular, the question remains as to whether
we can characterise $\DFAN$ using more natural subclass of pointwise
continuous functions without such a coding. 

In this paper, we consider a natural strengthening of the notion of
pointwise continuity which makes $\UCT$ equivalent to $\DFAN$.
Specifically, we strengthen the notion of continuity for functions
from $\UInt$ to $\Real$ by equipping them with more information about
their moduli of pointwise continuity.
The starting point of our work is Berger's result \cite{BergerFANandUC} that $\DFAN$ is equivalent to
the following statement:
\begin{description}
\item[($\UCc$)] Every continuous function $f \colon \Cantor \to
  \Nat$ with a continuous modulus is uniformly continuous.
\end{description}
Here, a \emph{modulus} of $f \colon \Cantor \to \Nat$ is  a function
$g \colon \Cantor \to \Nat$ such that \begin{equation}
\label{eq:Modulus} \forall \alpha, \beta \in \Cantor \left 
( \overline{\alpha}g(\alpha) =  \overline{\beta}g(\alpha) \imp
f(\alpha) = f(\beta)\right), \end{equation} where
$\overline{\alpha}n$ denotes the initial segment of $\alpha$ of length
$n$. 
In Section~\ref{sec:RealValedFunctionsOnCantor}, we show that $\UCc$ is equivalent to the following principle for
real-valued continuous functions:
\begin{description}
\item[($\UCTcC$)] Every continuous function $f \colon \Cantor \to
  \Real$ with a continuous modulus is uniformly continuous.
\end{description}

The above equivalence suggests that for a suitable notion of
continuous modulus for functions from $\UInt$ to $\Real$,
the following statement becomes equivalent to $\DFAN$:
\begin{description}
\item[($\UCTc$)] Every continuous function $f \colon \UInt \to \Real$ with
  a continuous modulus is uniformly continuous.
\end{description}
There are several possible choices for the notion of continuous
modulus of a function of the type $\UInt \to \Real$.
Here, we consider
a modulus of pointwise continuity of $f \colon \UInt \to \Real$
to
be an \emph{operation}~\cite[Chapter~2, Section~1]{Bishop-67}, i.e., a
function which does not necessarily respect the
equality on the domain. Specifically, a \emph{modulus} of a function $f \colon \UInt \to \Real$ is a family of functions
$g_{k} \colon \UInt \to \Nat$ for each $k \in \Nat$ such that for any
$k \in \Nat$ and $x,y \in \UInt$, it holds that 
\[
  |x - y| \leq 2^{-g_{k}(x)} \imp |f(x) - f(y)| \leq 2^{-k}.
\]
More precisely, $g_{k} \colon \UInt \to \Nat$ is a function
from \emph{the underlying set of regular sequences in 
$\UInt$} (cf.\ Section \ref{sec:RegularSequences}), so it only needs to respect pointwise equality of regular
sequences. By identifying the underlying set of regular sequences with
a subset of $\Baire$, we define such a modulus to be \emph{continuous}
if each $g_{k} \colon \UInt \to \Nat$ is pointwise continuous with
respect to the topology on $\Baire$. 
See Section~\ref{sec:ContinuousModulus} for the details.

Using the above notion of continuous modulus, we show that $\UCTc$ is
equivalent to $\DFAN$ (Section~\ref{sec:EquivalenceDFANUCTc}).  The
non-trivial part is deriving $\DFAN$ from $\UCTc$. Here, as in the
related works~\cite{PseudoCompUIntEquivUCT,Loeb05,ConvexFan}, we use
the Cantor discontinuum to construct a real-valued function on $\UInt$
from a bar of the binary fan. Our construction is similar to those in
\cite{PseudoCompUIntEquivUCT,Loeb05}, and in particular to
\cite{PseudoCompUIntEquivUCT}. However, our explicit treatment of real
numbers as regular sequences allows us give a more concrete
construction without relying on the Bishop's lemma~\cite[Chapter~4,
Lemma~3.8]{Bishop-Bridges-85}, which requires the countable choice.

The question remains whether the notion of functions from $\UInt$ to
$\Real$ having continuous moduli and Loeb's notion of continuous
functions are equivalent. The following observation suggests that
the answer would be positive.
Loeb's encoding of continuous functions from $\UInt$ to $\Real$
can be considered as a neighbourhood function~\cite[Chapter 4, Section
6.8]{ConstMathI} of some function $f \colon \UInt \to
\Real$. On the other hand, for a function of type $\Baire \to \Nat$, the existence of its continuous modulus is equivalent to
the existence of its neighbourhood function (see
Beeson~\cite[Chapter~VI, Section~8,
Exercise~8]{BeesonFoundationConstMath},
Kohlenbach~\cite[Proposition~4.4]{KohlenbachContMod}). Thus, it is natural
to expect that her notion of continuous function gives rise to a
continuous function from $\UInt$ to $\Real$ having a continuous modulus in
our sense. Indeed, we show the continuous real-valued functions
on the unit interval having continuous moduli are exactly those
functions induced by continuous functions described in \cite{Loeb05} (see Section~\ref{sec:Code}).

It should be noted that the equivalence of $\UCTc$ and
$\DFAN$ is immediate from \cite{Loeb05} once the
above mentioned equivalence between having continuous moduli and being
induced by Loeb's ``continuous functions'' is established.
Nevertheless, the latter equivalence is not so straightforward as it
requires us to use the quotient property of the intuitionistic
representation of the unit interval by the ternary spread. 
Moreover, our detailed proof of the equivalence of $\UCTc$ and $\DFAN$
in Section~\ref{sec:UCTc} without using some coding of continuous functions
would be more accessible to those who work in Bishop-style
constructive mathematics.

Throughout this paper, we work informally in Bishop-style constructive mathematics
\cite{Bishop-67}. However, one should be able to formalise our work in
Heyting arithmetic in all finite types $\HAw$
\cite[1.6.15]{Troelstra1973} with the axiom scheme $\qfAC^{1,0}$ of
quantifier free choice from sequences to numbers.

\begin{notation}
  \label{not:general}
  The letters $s,t,u$ range over the elements of finite binary
  sequences $\BSeq$ or finite ternary sequences $\TSeq$;
  the letters $\alpha,\beta,\gamma$ range over 
  the the elements of infinite sequences $\Cantor$ or $\TTree$.
  The set of finite binary (or ternary) sequences 
  of length $n \in \Nat$ is denoted by $\Bin^{n}$ (or $\Ter^{n}$).
  We write
  $\seq{x_{0}, \dots, x_{n-1}}$ for an element of $\BSeq$ (or
  $\TSeq$) of length $n$. The constant finite sequence of length
  $n$ with the value $i$ at each index is denoted by $i^{n}$.
  The length of $s$ is denoted by $\lh{s}$.
  The concatenation of $s$ and
  $t$ will be denoted by $s*t$, and the concatenation
  of a finite sequence $s$ and an infinite sequence $\alpha$ is
  denoted by $s * \alpha$. 
  We write $\alpha_{n}$ and $s_{n}$ ($n < \lh{s}$) for the value of
  $n$-th term of $\alpha$ and $s$.
  We write $s \preccurlyeq t$ if 
  $s$ is an initial segment of $t$. 
  We write $\overline{\alpha}n$ for the initial segment of
  $\alpha$ of length $n$ 
  and 
  $\overline{s}n$
  for the initial segment of $s$ of length $n \leq \lh{s}$.
  We write $\alpha \in s$ if
  $\overline{\alpha}\lh{s} = s$. 
  We write $\widehat{s}$ for $s * 0^{\omega}$ and $\widebreve{s}$
  for $s * 1^{\omega}$, where 
  $0^{\omega}$ and $1^{\omega}$ are infinite sequences of $0$ and
  $1$ respectively. 
\end{notation}

\section{Real numbers}\label{sec:Real}
As our standard notion of real numbers, we adopt Cauchy sequences of
rational numbers with explicitly given moduli, which we recall in Section \ref{sec:FundmentalSequences}.
For the purpose of this paper, however, it is sometimes convenient to
work with other (but equivalent) representations of real numbers; see Section \ref{sec:RegularSequences} and 
Section~\ref{sec:LeobReal}.

\subsection{Fundamental sequences}\label{sec:FundmentalSequences}
Among
the several possible choices of definition, we adopt the following
from Troelstra and van Dalen \cite[Chapter 5, Section 2]{ConstMathI}.
\begin{definition}
  A \emph{fundamental sequence with a modulus} is a sequence
  $\seq{r_{n}}_{n \in
  \Nat}$ of rational numbers together with a function $\delta \colon \Nat \to \Nat$,
  called a \emph{modulus} of $\seq{r_{n}}_{n \in \Nat}$, such that
  \[
    \forall k,n,m \in \Nat
    \left( |r_{\delta(k) + n} - r_{\delta(k) + m}| \leq 2^{-k} \right).
  \]
  Two fundamental sequences $\seq{r_{n}}_{n \in \Nat}$ and
  $\seq{q_{n}}_{n \in \Nat}$ with moduli $\delta$ and $\xi$
  respectively are \emph{equal}, written $\seq{r_{n}}_{n \in \Nat}
  \simeq \seq{q_{n}}_{n \in \Nat}$, if
  \begin{equation}
    \label{def:FundamentalEqual}
    \forall k \exists n \forall m \left( |r_{n + m} - q_{n + m}|
    \leq
    2^{-k}\right).
  \end{equation}
  By a \emph{real number}, we mean a fundamental sequence 
  with a modulus. 
\end{definition}

  The orders on real numbers are defined by
  \begin{align*}
    \seq{r_{n}}_{n \in \Nat} < \seq{q_{n}}_{n \in \Nat}
    &\defeqiv \exists k, n \in \Nat \forall m \in \Nat \left( 
    q_{n+m} - r_{n+m} > 2^{-k} \right), \\
    \seq{r_{n}}_{n \in \Nat} \leq \seq{q_{n}}_{n \in \Nat}
    &\defeqiv 
    \neg \left( \seq{q_{n}}_{n \in \Nat} < \seq{r_{n}}_{n \in \Nat}
    \right).
  \end{align*}
  One can show that these orders are well-defined with respect to
  $\simeq$ and that 
  \[
    \seq{r_{n}}_{n \in \Nat} \simeq \seq{q_{n}}_{n \in \Nat}
    \leftrightarrow \seq{r_{n}}_{n \in \Nat} \leq \seq{q_{n}}_{n \in \Nat}
    \wedge \seq{q_{n}}_{n \in \Nat} \leq \seq{r_{n}}_{n \in \Nat}.
  \]
  Rational numbers are embedded into fundamental sequences by
  $r \mapsto \seq{r}_{n \in \Nat}$,
  which is an order embedding.

  The arithmetical operations on fundamental sequences are 
  defined term-wise: if $\seq{r_{n}}_{n \in \Nat}$ 
  and $\seq{q_{n}}_{n \in \Nat}$ are fundamental sequences
  with moduli $\delta$ and $\xi$ respectively, then
\begin{align*}
    \seq{r_{n}}_{n \in \Nat} + \seq{q_{n}}_{n \in \Nat}
    &\defeql
    \seq{r_{n} + q_{n}}_{n \in \Nat} &&\text{with modulus $\zeta(k) =
    \max\left\{\delta(k+1),\xi(k+1)  \right\}$},\\
    -\seq{r_{n}}_{n \in \Nat} 
    &\defeql
    \seq{-r_{n}}_{n \in \Nat} &&\text{with modulus $\delta$},\\
    |\seq{r_{n}}_{n \in \Nat}|
    &\defeql
    \seq{|r_{n}|}_{n \in \Nat} &&\text{with modulus $\delta$}.
\end{align*}

\subsection{Regular sequences}\label{sec:RegularSequences}
For the spread representation of the unit interval in Section
\ref{sec:SpreadRepresentation}, it is convenient to work with regular
sequences.\footnote{Our terminology differs from
  Bishop~\cite[Chapter 2]{Bishop-67} in which a regular sequence is
defined by the property 
$
  \forall n,m \in \Nat
  \left( |r_{n} - r_{m}| \leq 1/n + 1/m  \right).
$}
\begin{definition}\label{def:RegularSequences}
A sequence $\seq{r_{n}}_{n \in \Nat}$ of rational numbers is 
\emph{regular} if 
\[
  \forall n \in \Nat
  \left( |r_{n} - r_{n+1}| \leq 2^{-(n+1)}\right).
\]
Two regular sequences $\seq{r_{n}}_{n \in \Nat}$ and  $\seq{q_{n}}_{n \in \Nat}$
are defined to be \emph{equal}, written  
$\seq{r_{n}}_{n \in \Nat} \simeq \seq{q_{n}}_{n \in \Nat}$, if
\begin{equation}
  \label{def:RegularEqual}
  \forall n \in \Nat \left( |r_{n+1} - q_{n+1}| \leq 2^{-n} \right).
\end{equation}
\end{definition}
The orders on regular sequences are defined by
\begin{align*}
  \seq{r_{n}}_{n \in \Nat} < \seq{q_{n}}_{n \in \Nat}
  &\defeqiv \exists  n \in \Nat \left( q_{n+1} - r_{n+1} > 2^{-n}\right), \\
  \seq{r_{n}}_{n \in \Nat} \leq \seq{q_{n}}_{n \in \Nat}
  &\defeqiv 
  \neg \left(\seq{q_{n}}_{n \in \Nat} < \seq{r_{n}}_{n \in
  \Nat}  \right).
\end{align*}
Note that
$\seq{r_{n}}_{n \in \Nat} \simeq \seq{q_{n}}_{n \in \Nat}
\leftrightarrow \seq{r_{n}}_{n \in \Nat} \leq \seq{q_{n}}_{n \in \Nat}
\wedge \seq{q_{n}}_{n \in \Nat} \leq \seq{r_{n}}_{n \in \Nat}$.
Also, it is straightforward to show that
\[
  \seq{r_{n}}_{n \in \Nat} \leq \seq{q_{n}}_{n \in \Nat}
  \leftrightarrow \forall k \in \Nat \exists n \in \Nat \forall m \in \Nat
  \left( r_{n+m} - q_{n+m} \leq 2^{-k} \right),
\]
and so
\begin{equation}
  \label{eq:RegEquality}
  \seq{r_{n}}_{n \in \Nat} \simeq \seq{q_{n}}_{n \in \Nat}
  \leftrightarrow
  \forall k \in \Nat \exists n \in \Nat \forall m \in \Nat
  |r_{n+m} - q_{n+m}| \leq 2^{-k}.
\end{equation}
  The arithmetical operations on regular sequences are 
  defined by
\begin{align*}
    \seq{r_{n}}_{n \in \Nat} + \seq{q_{n}}_{n \in \Nat}
    &\defeql
    \seq{r_{n+1} + q_{n+1}}_{n \in \Nat}, \\
    -\seq{r_{n}}_{n \in \Nat} 
    &\defeql
    \seq{-r_{n}}_{n \in \Nat}, \\
    |\seq{r_{n}}_{n \in \Nat}|
    &\defeql
    \seq{|r_{n}|}_{n \in \Nat}. 
\end{align*}
\begin{lemma}
  \label{lem:RegularOrder}
  For any regular sequences $\seq{r_{n}}_{n \in \Nat}$ and 
  $\seq{q_{n}}_{n \in \Nat}$, we have
  \[
  \seq{r_{n}}_{n \in \Nat} < \seq{q_{n}}_{n \in \Nat}
  \leftrightarrow
  \exists k,n \in \Nat \forall m \in \Nat \left( q_{n+m} -
  r_{n+m} > 2^{-k}\right).
  \]
\end{lemma}
\begin{proof}
  ($\Rightarrow$) Suppose that 
  $\seq{r_{n}}_{n \in \Nat} < \seq{q_{n}}_{n \in \Nat}$. 
  Then, there exists $k \in \Nat$ such that $q_{k+1} - r_{k+1} >
  2^{-k}$. Choose $l \in \Nat$ such that 
  $q_{k+1} - r_{k+1} > 2^{-k} + 2^{-l}$. For any $m \in \Nat$,
  \[
    \begin{aligned}
      q_{k+1 + m } - r_{k+1 + m}
      &= 
      q_{k+1 + m } - q_{k+1} + q_{k+1} - r_{k+1} + r_{k+1}- r_{k+1 + m} \\
      &> -2^{-(k+1)} + (2^{-k} + 2^{-l}) - 2^{-(k+1)}  \\
      &= 2^{-l}.
    \end{aligned}
  \]
   \noindent($\Leftarrow$) Suppose there are $k,n \in \Nat$ such that $\forall m
  \in \Nat \left(q_{n+m} - r_{n+m} > 2^{-k}\right)$. Put $M = \max
  \left\{ k, n \right\}$. Then
    $
    q_{M+1} - r_{M+1} > 2^{-k} \geq 2^{-M}.
    $
\end{proof}

\begin{proposition}
  \label{prop:OrderBiject}
  There exists an order preserving bijection between the set of fundamental sequences
  with moduli equipped with the equality \eqref{def:FundamentalEqual}
  and the set of regular sequences equipped with the equality
  \eqref{def:RegularEqual}: 
  \begin{enumerate}
    \item If $\seq{r_{n}}_{n \in \Nat}$ is a fundamental sequence with
      modulus $\delta$, then $\seq{r_{\delta(n+1)}}_{n \in \Nat}$ is
      a regular sequence.
    \item If $\seq{r_{n}}_{n \in \Nat}$ is a regular sequence, then it
      is a fundamental sequence with modulus $k \mapsto k$.
  \end{enumerate}
  Moreover, the bijection preserves arithmetical operations up to
  equalities on fundamental sequences and regular sequences.
\end{proposition}
\begin{proof}
  If $\seq{r_{n}}_{n \in \Nat}$ is a fundamental sequence with modulus
  $\delta$, then 
  \[
    |r_{\delta(n+1)} - r_{\delta(n+2)}| \leq 2^{-(n+1)}
  \]
  for all $n \in \Nat$, so $\seq{r_{\delta(n+1)}}_{n \in \Nat}$ is a
  regular sequence.
  Conversely, if $\seq{r_{n}}_{n \in \Nat}$ is a regular sequence, then 
  \[
    |r_{k+n} - r_{k+n+m}| \leq 2^{-(k+n)}
  \]
  for all $k, n,m \in \Nat$. Thus, $\seq{r_{n}}_{n \in \Nat}$ 
  is a fundamental sequence with modulus $k \mapsto k$

  Let $\seq{r_{n}}_{n \in \Nat}$ and 
  $\seq{q_{n}}_{n \in \Nat}$ be fundamental sequences with
  moduli $\delta$ and $\xi$ respectively. We show that 
  \begin{equation}
  \label{prop:OrderBiject1}
    \seq{r_{n}}_{n \in \Nat} \simeq \seq{q_{n}}_{n \in \Nat} 
    \leftrightarrow
    \seq{r_{\delta(n+1)}}_{n \in \Nat} \simeq \seq{q_{\xi(n+1)}}_{n
    \in \Nat},
  \end{equation}
  where the left hand side is the equality of fundamental sequences
  and the right hand side is that of regular sequences.
  By \eqref{eq:RegEquality}, it suffice to show that
  \begin{equation}
  \label{prop:OrderBiject2}
    \seq{r_{n}}_{n \in \Nat} \simeq \seq{r_{\delta(n+1)}}_{n \in \Nat}
  \end{equation}
  as fundamental sequences, i.e.,
    $
    \forall k \in \Nat \exists n \in \Nat \forall m \in \Nat 
    |r_{n+m} - r_{\delta(n+1)+m} | \leq 2^{-k}.
    $
  Let $k \in \Nat$, and put $n = \max\left\{\delta(k), k \right\}$.
  Fix $m \in \Nat$.  If $n \geq \delta(n+1)$, then
  \[
    |r_{n+m}  - r_{\delta(n+1) + m} |
    \leq 2^{-(n+1)}
    \leq 2^{-(k+1)} < 2^{-k}.
  \]
  If $n \leq \delta(n+1)$, then $\delta(k) \leq \delta(n+1)$ and
  $\delta(k) \leq n$, so $|r_{n+m} - r_{\delta(n+1) + m}| \leq
  2^{-k}$.
  
  Let $\Real$ and $\Real'$ be the sets of fundamental sequences 
  and regular sequences, respectively, with respective equalities.
  Write $F \colon \Real \to \Real'$ and $G \colon \Real' \to \Real$
  for the mappings
  $\seq{r_{n}}_{n \in \Nat} \mapsto \seq{r_{\delta(n+1)}}_{n \in
  \Nat}$ and $\seq{r_{n}}_{n \in \Nat} \mapsto \seq{r_{n}}_{n \in
  \Nat}$, respectively.
  By \eqref{prop:OrderBiject1} and \eqref{eq:RegEquality}, these
  mappings are well-defined. 
  We have $F \circ G = \id_{\Real'}$, and we also have $G \circ F = \id_{\Real}$
  by \eqref{prop:OrderBiject2}. By Lemma \ref{lem:RegularOrder}, we
  have
  \[
    \seq{r_{n}}_{n \in \Nat} < \seq{q_{n}}_{n \in \Nat} 
    \leftrightarrow G( \seq{r_{n}}_{n \in \Nat} )
    < G( \seq{q_{n}}_{n \in \Nat} )
  \]
  for regular sequences $\seq{r_{n}}_{n \in \Nat}$ and $\seq{q_{n}}_{n
  \in \Nat}$. Thus
  \begin{align*}
    \seq{r_{n}}_{n \in \Nat} < \seq{q_{n}}_{n \in \Nat} 
    &\leftrightarrow G \circ F( \seq{r_{n}}_{n \in \Nat} )
    < G \circ F( \seq{q_{n}}_{n \in \Nat} )\\
    &\leftrightarrow F( \seq{r_{n}}_{n \in \Nat} )
    < F( \seq{q_{n}}_{n \in \Nat} )
  \end{align*}
  for fundamental sequences $\seq{r_{n}}_{n \in \Nat}$ and
  $\seq{q_{n}}_{n \in \Nat}$ (with some moduli).
  Thus $F$ and $G$ are order bijections. It is then straightforward to
  show that $F$ and $G$ preserve arithmetical operations.
\end{proof}
Proposition \ref{prop:OrderBiject} allows us
to freely use fundamental sequences with moduli or
regular sequences as the notion of real numbers whichever is
convenient, and we will do so without explicit reference to 
the proposition.
The set of real numbers is denoted by $\Real$ and 
the equality on $\Real$ will be written $\simeq$.
The unit interval $\UInt$ is defined as usual: $\UInt = \left\{ x \in
  \Real \mid 0 \leq x \leq 1 \right\}$.

\subsection{Shrinking sequences of rational intervals}
\label{sec:LeobReal}
Loeb~\cite{Loeb05} uses a representation of real numbers by shrinking
sequences of rational intervals. 
To relate our work with \cite{Loeb05} in Section \ref{sec:Code}, we
briefly recall some basic definitions about this representation.

\begin{notation}
  \label{not:Interval}
Let
\begin{equation}
  \label{def:RatInt}
  \RatInt \defeql \left\{(p,q) \in \Rat \times \Rat \mid p \leq q \right\}, 
\end{equation}
which can be thought of as the set of closed intervals $[p,q]$ with rational
endpoints $p$ and $q$.
For $(p,q) \in \RatInt$, let $\lh{(p,q)} = q - p$, the \emph{length}
of the interval $[p,q]$. For  $\mathbb{I}, \mathbb{J}\in \RatInt$,
define
  \begin{align*}
    \fst{\mathbb{I}} &\defeql \text{the first projection of
      $\mathbb{I}$},\\
    \snd{\mathbb{I}} &\defeql \text{the second projection of
      $\mathbb{I}$},\\
    \mathbb{I} \sqsubseteq \mathbb{J}
    &\defeqiv \fst{\mathbb{J}} \leq \fst{\mathbb{I}} \wedge
    \snd{\mathbb{I}} \leq \snd{\mathbb{J}},  \\
    \mathbb{I} \approx \mathbb{J} 
    &\defeqiv
    \fst{\mathbb{J}} \leq \snd{\mathbb{I}} \wedge
    \fst{\mathbb{I}} \leq \snd{\mathbb{J}}. 
  \end{align*}
\end{notation}
The following notion is called a \emph{real number} in Loeb
\cite[Section 3]{Loeb05}.
\begin{definition}
  \label{def:LoebReal}
  A sequence $\seq{\RInt_n}_{n \in \Nat} \colon \Nat \to \RatInt$ is
  a \emph{shrinking sequence} if 
  \begin{enumerate}[({S}1)]
    \item\label{def:LoebRea1} $\forall n \in \Nat \left( \RInt_{n+1} \sqsubseteq \RInt_{n}
      \right)$,
    \item\label{def:LoebRea2} $\forall k \in \Nat \exists n \in \Nat \left( \lh{\RInt_n}
      \leq 2^{-k} \right)$.
  \end{enumerate}
  Two shrinking sequences
  $\seq{\RInt_n}_{n \in \Nat}$
  and $\seq{\mathbb{J}_n}_{n \in \Nat}$
  are \emph{equal} if
  \begin{equation}
    \label{def:LoebRealEqual}
    \forall n \in \Nat \left( \RInt_n \approx \mathbb{J}_n \right).
  \end{equation}
  The set of shrinking sequences with the equality \eqref{def:LoebRealEqual} is denoted by
  $\LoebReal$.

  The orders on $\LoebReal$  are defined by
  \begin{align*}
    \seq{\RInt_n}_{n \in \Nat} < \seq{\mathbb{J}_n}_{n \in \Nat}
    &\defeqiv 
    \exists n \in \Nat
    \left( \snd{\RInt_n} < \fst{\mathbb{J}_n}\right),\\
    \seq{\RInt_n}_{n \in \Nat} \leq \seq{\mathbb{J}_n}_{n \in \Nat}
    &\defeqiv 
    \neg \left( \seq{\mathbb{J}_n}_{n \in \Nat} < \seq{\RInt_n}_{n \in \Nat}
    \right).
  \end{align*}
  These orders are well-defined with respect to the equality on $\LoebReal$.
  The arithmetical operations on $\LoebReal$ are 
  defined by 
\begin{align*}
  \seq{\mathbb{I}_{n}}_{n \in \Nat} + \seq{\mathbb{J}_{n}}_{n \in \Nat}
    &\defeql
    \seq{(\fst{\mathbb{I}_{n}} + \fst{\mathbb{J}_{n}},
    \snd{\mathbb{I}_{n}} + \snd{\mathbb{J}_{n}})}_{n \in \Nat},\\
    -\seq{\mathbb{I}_{n}}_{n \in \Nat} 
    &\defeql
    \seq{\left(-\snd{\mathbb{I}_{n}},-\fst{\mathbb{I}_{n}}\right)}_{n
    \in \Nat},\\
    |\seq{\mathbb{I}_{n}}_{n \in \Nat}|
    &\defeql
    \seq{\left(\max \left\{ -\snd{\mathbb{I}_{n}},
    \fst{\mathbb{I}_{n}} \right\},
    \max \left\{  -\fst{\mathbb{I}_{n}},
    \snd{\mathbb{I}_{n}} \right\}\right)}_{n \in \Nat}.
\end{align*}
\end{definition}
\begin{proposition}
  \label{prop:OrderBijectLoebReal}
  There exists an order preserving bijection between the set of shrinking
  sequences equipped with the equality \eqref{def:LoebRealEqual} and the set of
  regular sequences equipped with the equality \eqref{def:RegularEqual}: 
  \begin{enumerate}
    \item If $\seq{\RInt_n}_{n \in \Nat}$ is a shrinking sequence, then
      $\seq{\fst{\RInt_{\delta(n)}}}_{n \in \Nat}$
      is a regular sequence, where
      \[
        \delta(k) \defeql \text{the least $n \in \Nat$ such that
          $\lh{\RInt_{n}} \leq 2^{-(k+1)}$}.
      \]
    \item If $\seq{r_{n}}_{n \in \Nat}$ is a regular sequence, then
      $\seq{(r_{n+1} - 2^{-(n+1)}, r_{n+1} + 2^{-(n+1)})}_{n \in \Nat}$
      is a shrinking sequence.
  \end{enumerate}
  Moreover, the bijection preserves arithmetical operations up to
  equalities on shrinking sequences and regular sequences.
\end{proposition}
\begin{proof}
  Routine.
\end{proof}

\section{Real-valued functions on the Cantor space}
\label{sec:RealValedFunctionsOnCantor}
In this section, we extend the equivalence between the decidable fan
theorem ($\DFAN$) and the uniform continuity principle with continuous modulus
($\UCc$) due to Berger~\cite{BergerFANandUC} to real-valued functions
on the Cantor space $\Cantor$ (cf.\
Introduction~\ref{sec:Introduction}).

First, we recall the notion of continuity on the Cantor space (see Troelstra and van Dalen~\cite[Chapter 4,
1.5]{ConstMathI}).
\begin{definition}
  \label{def:ContonCantor}
  \leavevmode
  \begin{enumerate}
    \item 
  A function $f \colon \Cantor \to \Nat$ is \emph{pointwise continuous} if
  \[
    \forall \alpha \in \Cantor \exists n \in \Nat \forall \beta \in
    \Cantor 
    \left( 
    \overline{\alpha}n = \overline{\beta}n \imp f(\alpha) = f(\beta)
    \right),
  \]
  and \emph{uniformly continuous} if
  \[
    \exists n \in \Nat \forall \alpha, \beta \in \Cantor
    \left( 
    \overline{\alpha}n = \overline{\beta}n \imp f(\alpha) = f(\beta)
    \right).
  \]

  \item
  A function $f \colon \Cantor \to \Real$ is \emph{pointwise
  continuous} if
  \[
    \forall \alpha \in \Cantor\forall k \in \Nat \exists n \in \Nat
    \forall \beta \in \Cantor \left( \overline{\alpha}n =
    \overline{\beta}n \imp |f(\alpha) - f(\beta)| \leq 2^{-k} \right),
  \]
  and \emph{uniformly continuous} if there exists $\omega \colon \Nat
  \to \Nat$, called a \emph{modulus of uniform continuity}, such that
  \[
    \forall k \in \Nat
    \forall \alpha, \beta \in \Cantor \left(
    \overline{\alpha}\omega(k)  = \overline{\beta}\omega(k) \imp
    |f(\alpha) - f(\beta) | \leq 2^{-k}\right).
  \]
  \end{enumerate}
  Unless otherwise noted, \emph{continuous} means pointwise
  continuous in this paper.
\end{definition}
\begin{remark}
  \label{rem:UnifContOnCantor}
  If $f \colon \Cantor \to \Nat$ is uniformly continuous, then
  there is a \emph{least modulus of uniform continuity} of $f$,
  i.e., there is a least $N \in \Nat$ such that 
  \begin{equation}
    \label{eq:UnifCont}
    \forall \alpha, \beta \in \Cantor
    \left( 
    \overline{\alpha}N = \overline{\beta}N \imp f(\alpha) = f(\beta)
    \right).
  \end{equation}
  Specifically, if $f \colon \Cantor \to \Nat$ is uniform
  continuous, then there exists $N \in \Nat$ which satisfies
  \eqref{eq:UnifCont}. Then
  \begin{align*}
    L \defeql
    \text{
      the least $n \leq N$ such that
      $\forall a \in \Bin^{n} \forall b \in \Bin^{N-n} \left(
      f(\widehat{a}) = f(\widehat{a * b})
      \right)$
    }
  \end{align*}
  is the least modulus of uniform continuity of $f$.
\end{remark}

\begin{definition}
  A \emph{modulus} of $f \colon \Cantor
\to \Real$ is  a function $g \colon \Nat \to \Cantor \to \Nat$ such that 
\begin{equation}
  \label{eq:ModulusRealValued}
  \forall k \in \Nat \forall \alpha, \beta \in \Cantor
  \left( \overline{\alpha}g_{k}(\alpha) =  \overline{\beta}g_{k}(\alpha)
  \imp |f(\alpha) - f(\beta)| \leq 2^{-k} \right).
\end{equation}
A modulus $g$ of $f \colon \Cantor
\to \Real $ is \emph{continuous} if $g_{k} \colon \Cantor \to \Nat$
is pointwise continuous for each $k \in \Nat$.

The principle $\UCTcC$ is the following statement:%
\begin{description}
 \item[($\UCTcC$)] 
  Every continuous function $f \colon \Cantor \to \Real$ 
  with a continuous modulus is uniformly continuous. 
\end{description}
\end{definition}

\begin{lemma}
  \label{lem:ContModItself}
  If $f \colon \Cantor \to \Real$ has a continuous
  modulus, then $f$ has a continuous modulus $g \colon \Nat \to
  \Cantor \to \Nat$ such that $g_{k}$ is a modulus of itself for each
  $k \in \Nat$.
\end{lemma}
\begin{proof}
  Let $g \colon \Nat \to \Cantor \to \Nat$ be a continuous 
  modulus of  $f \colon \Cantor \to
  \Real$.  For each $k \in \Nat$,
  define $G_{k} \colon \Cantor \to \Nat$ by
  \[
    G_{k}(\alpha) \defeql \text{the least $n$ such that
      $g_{k+1}(\widehat{\overline{\alpha}n}) < n$}.
  \]
  Note that $G_{k}$ is well-defined because $g_{k+1}$ is
  continuous.
  It is straightforward to show that 
  for each $k$, the function $G_{k}$ is a continuous modulus
  of itself (see \cite[Lemma 2.2]{FujiwaraKawaiBRCC}). 
  We show that  $G$ is a modulus of $f$. Let $k \in \Nat$ and $\alpha,
  \beta \in \Cantor$, and suppose that
  $\overline{\alpha}G_{k}(\alpha) = \overline{\beta}G_{k}(\alpha)$.
  Since $g_{k+1}(\widehat{\overline{\alpha}G_{k}(\alpha)}) < G_{k}(\alpha)$, we
  have
  \begin{align*}
    \overline{\alpha}g_{k+1}(\widehat{\overline{\alpha}G_{k}(\alpha)})
    &=
    \overline{\widehat{\left(\overline{\alpha}G_{k}(\alpha)
  \right)}}g_{k+1}(\widehat{\overline{\alpha}G_{k}(\alpha)}) \\
    &= 
    \overline{\widehat{\left(\overline{\beta}G_{k}(\alpha)  \right)}}g_{k+1}(\widehat{\overline{\alpha}G_{k}(\alpha)})
    =
    \overline{\beta}g_{k+1}(\widehat{\overline{\alpha}G_{k}(\alpha)}).
  \end{align*}
    Thus
    \[
      \begin{aligned}
        |f(\alpha) - f(\beta)| 
        &\leq |f(\alpha) - f(\widehat{\overline{\alpha}G_{k}(\alpha)}) | +
        | f(\widehat{\overline{\alpha}G_{k}(\alpha)}) - f(\beta)|\\
        &\leq 2^{-(k+1)} + 2^{-(k+1)} = 2^{-k}.
      \end{aligned}
    \]
  Hence, $G$ is a modulus of $f$.
\end{proof}
\begin{proposition}\label{prop:ContMod}
  The principle $\UCTcC$ is equivalent to $\UCc$.
\end{proposition}
\begin{proof}
  $\UCTcC$ obviously implies $\UCc$.
  For the converse,
  let $g \colon \Nat \to \Cantor \to \Nat$ be a continuous modulus of a function $f \colon \Cantor \to
  \Real$. By Lemma \ref{lem:ContModItself}, we may assume that
  $g_{k}$ is a modulus of itself for each $k \in \Nat$,
  and hence $g_{k}$ is uniformly continuous for each $k \in \Nat$ by $\UCc$.

  For each $k \in \Nat$, let $N_{k}$ be the least modulus of uniform
  continuity of $g_{k+1}$ (see Remark \ref{rem:UnifContOnCantor}). Put
  \[
    \omega(k) \defeql \max \left\{ g_{k+1}(\widehat{s}) \mid  s \in
    \Bin^{N_{k}} \right\}.
  \]
  Then, for any $\alpha, \beta \in \Cantor$ such that
  $\overline{\alpha}\omega(k) = \overline{\beta}\omega(k)$, we have 
  $\overline{\widehat{\overline{\alpha}\omega(k)}}g_{k+1}(\alpha) =
  \overline{\alpha}g_{k+1}(\alpha)$ and
  $\overline{\widehat{\overline{\beta}\omega(k)}}g_{k+1}(\beta) =
  \overline{\beta}g_{k+1}(\beta)$. Thus
  \[
    \begin{aligned}
      |f(\alpha) - f(\beta)| 
      &\leq |f(\alpha) - f(\widehat{\overline{\alpha}\omega(k)})|
      + |f(\beta) - f(\widehat{\overline{\beta}\omega(k)})| \\
      &\leq 2^{-(k+1)} + 2^{-(k+1)} 
      = 2^{-k}.
    \end{aligned}
  \]
  Therefore $f$ is uniformly continuous with modulus $\omega$.
\end{proof}

\section{Spread representation of the unit interval}
\label{sec:SpreadRepresentation}
We review some basic properties of the spread representation of the
unit interval $\UInt$ (Troelstra and van Dalen~\cite[Chapter 6,
Section 3]{ConstMathI}). Specifically, we use the representation
described in Loeb~\cite{Loeb05}, which is slightly different from the
one described in \cite{ConstMathI}.  This representation is
essentially the signed-digit representation of the unit interval, where
each real number in the unit interval is represented by a path in the
ternary tree (see e.g., Lubarsky and
Richman~\cite{SignedBitLubarskyRichman}).  For the reader's
convenience, we describe this representation in some detail. 

\begin{figure}[tb]
  \centering
\begin{tikzpicture}[xscale=0.8,yscale=1.3]
%  \draw  (-4,0) -- (4,0);
  %
  % 1st level
  %
  \draw (0,0) -- (-4,1);
  \draw (0,0) -- (0,1);
  \draw (0,0) -- (4,1);
  %
  % 2nd level
  %
  \draw (-4,1) -- (-6,2);
  \draw (-4,1) -- (-4,2);
  \draw (-4,1) -- (-2,2);
  \draw (0,1) -- (-2,2);
  \draw (0,1) -- (0,2);
  \draw (0,1) -- (2,2);
  \draw (4,1) -- (2,2);
  \draw (4,1) -- (4,2);
  \draw (4,1) -- (6,2);
  %
  % 3rd level
  %
  \draw (-6,2) -- (-7,3);
  \draw (-6,2) -- (-6,3);
  \draw (-6,2) -- (-5,3);
  \draw (-4,2) -- (-5,3);
  \draw (-4,2) -- (-4,3);
  \draw (-4,2) -- (-3,3);
  \draw (-2,2) -- (-3,3);
  \draw (-2,2) -- (-2,3);
  \draw (-2,2) -- (-1,3);
  \draw (0,2) -- (-1,3);
  \draw (0,2) -- (0,3);
  \draw (0,2) -- (1,3);
  \draw (2,2) -- (1,3);
  \draw (2,2) -- (2,3);
  \draw (2,2) -- (3,3);
  \draw (4,2) -- (3,3);
  \draw (4,2) -- (4,3);
  \draw (4,2) -- (5,3);
  \draw (6,2) -- (5,3);
  \draw (6,2) -- (6,3);
  \draw (6,2) -- (7,3);
  %
  % 4rd level
  %
  \draw (-7,3) -- (-7.5,4);
  \draw (-7,3) -- (-7,4);
  \draw (-7,3) -- (-6.5,4);
  \draw (-6,3) -- (-6.5,4);
  \draw (-6,3) -- (-6,4);
  \draw (-6,3) -- (-5.5,4);
  \draw (-5,3) -- (-5.5,4);
  \draw (-5,3) -- (-5,4);
  \draw (-5,3) -- (-4.5,4);
  \draw (-4,3) -- (-4.5,4);
  \draw (-4,3) -- (-4,4);
  \draw (-4,3) -- (-3.5,4);
  \draw (-3,3) -- (-3.5,4);
  \draw (-3,3) -- (-3,4);
  \draw (-3,3) -- (-2.5,4);
  \draw (-2,3) -- (-2.5,4);
  \draw (-2,3) -- (-2,4);
  \draw (-2,3) -- (-1.5,4);
  \draw (-1,3) -- (-1.5,4);
  \draw (-1,3) -- (-1,4);
  \draw (-1,3) -- (-0.5,4);
  \draw (0,3) -- (-0.5,4);
  \draw (0,3) -- (0,4);
  \draw (0,3) -- (0.5,4);
  \draw (1,3) -- (0.5,4);
  \draw (1,3) -- (1,4);
  \draw (1,3) -- (1.5,4);
  \draw (2,3) -- (1.5,4);
  \draw (2,3) -- (2,4);
  \draw (2,3) -- (2.5,4);
  \draw (3,3) -- (2.5,4);
  \draw (3,3) -- (3,4);
  \draw (3,3) -- (3.5,4);
  \draw (4,3) -- (3.5,4);
  \draw (4,3) -- (4,4);
  \draw (4,3) -- (4.5,4);
  \draw (5,3) -- (4.5,4);
  \draw (5,3) -- (5,4);
  \draw (5,3) -- (5.5,4);
  \draw (6,3) -- (5.5,4);
  \draw (6,3) -- (6,4);
  \draw (6,3) -- (6.5,4);
  \draw (7,3) -- (6.5,4);
  \draw (7,3) -- (7,4);
  \draw (7,3) -- (7.5,4);
  %
  %%%%%%%%%%%%%%%%%%%%
  % Numbering
  %%%%%%%%%%%%%%%%%%%%
  %
  % Root
  \draw [] (0,0)  node[inner sep=2pt, fill=white] {$1$};
  %
  % First level
  %
  \draw [] (-4,1) node[inner sep=2pt, fill=white] {$1$};
  \draw [] (0,1)  node[inner sep=2pt, fill=white] {$2$};
  \draw [] (4,1)  node[inner sep=2pt, fill=white] {$3$};
  %
  % 2nd level
  %
  \draw [] (-6,2)   node[inner sep=2pt, fill=white] {$1$};
  \draw [] (-4,2)   node[inner sep=2pt, fill=white] {$2$};
  \draw [] (-2,2)   node[inner sep=2pt, fill=white] {$3$};
  \draw [] (0,2)    node[inner sep=2pt, fill=white] {$4$};
  \draw [] (2,2)    node[inner sep=2pt, fill=white] {$5$};
  \draw [] (4,2)    node[inner sep=2pt, fill=white] {$6$};
  \draw [] (6,2)    node[inner sep=2pt, fill=white] {$7$};
  %
  % 3rd level
  %
  \draw [] (-7,3)   node[inner sep=2pt, fill=white] {$1$};
  \draw [] (-6,3)   node[inner sep=2pt, fill=white] {$2$};
  \draw [] (-5,3)   node[inner sep=2pt, fill=white] {$3$};
  \draw [] (-4,3)   node[inner sep=2pt, fill=white] {$4$};
  \draw [] (-3,3)   node[inner sep=2pt, fill=white] {$5$};
  \draw [] (-2,3)   node[inner sep=2pt, fill=white] {$6$};
  \draw [] (-1,3)   node[inner sep=2pt, fill=white] {$7$};
  \draw [] (0,3)    node[inner sep=2pt, fill=white] {$8$};
  \draw [] (1,3)    node[inner sep=2pt, fill=white] {$9$};
  \draw [] (2,3)    node[inner sep=2pt, fill=white] {$10$};
  \draw [] (3,3)    node[inner sep=2pt, fill=white] {$11$};
  \draw [] (4,3)    node[inner sep=2pt, fill=white] {$12$};
  \draw [] (5,3)    node[inner sep=2pt, fill=white] {$13$};
  \draw [] (6,3)    node[inner sep=2pt, fill=white] {$14$};
  \draw [] (7,3)    node[inner sep=2pt, fill=white] {$15$};
\end{tikzpicture}
\caption{The numbering of the nodes of $\TSeq$} \label{fig:Numbering}
\end{figure}

Consider the ternary tree $\TSeq$.
To each node $s \in \TSeq$ we assign a number $N(s)$ as follows (see Figure \ref{fig:Numbering}):
\begin{align*}
  N(\nil) &\defeql  1,&\quad
  N(s * \seq{i}) &\defeql 2 N(s) + (i - 1) \quad (i \in \Ter).
\end{align*}
Each path $\alpha$ in $\TSeq$ (i.e., $\alpha \in \TTree$)
determines a regular sequence
$x_{\alpha}$ in $\UInt$ by
\[
  x_{\alpha} \defeql \seq{2^{-(n+1)} N(\overline{\alpha}
n)}_{n \in \Nat}.
\]
Write $x^{n}_{\alpha}$ for the $n$-th term of $x_{\alpha}$, i.e., 
\[
  x_{\alpha}^{n} \defeql 2^{-(n+1)} N(\overline{\alpha} n).
\]
\begin{lemma}
  \label{lem:RepresentationApprox}
  For each $\alpha \in \TTree$ and $n \in \Nat$, we have
  $
  |x_{\alpha} - x^{n}_{\alpha} | \leq 2^{-(n+1)}.
  $
\end{lemma}
\begin{proof}
Note that
\begin{align*}
  |x^{n}_{\alpha} - x^{n+1}_{\alpha}|
  &= 2^{-(n+2)}|2N(\overline{\alpha}n) - N(\overline{\alpha}(n+1))|
  \\
  &= 2^{-(n+2)}|2N(\overline{\alpha}n) - (2N(\overline{\alpha}n) +
  (\alpha(n) - 1))| \\
  &\leq 2^{-(n+2)}.
\end{align*}
Hence $|x^{n}_{\alpha} - x^{n+m}_{\alpha}| < 2^{-(n+1)}$ for all
$n,m \in \Nat$. Thus
  $
    |x_{\alpha} - x^{n}_{\alpha} | \leq 2^{-(n+1)}
  $
for all $n \in \Nat$.
\end{proof}

\begin{corollary}
\label{cor:UnifContPhi}
The function $\Phi \colon \TTree \to \UInt$ defined by 
\[
  \Phi(\alpha) \defeql x_{\alpha}
\]
is uniformly continuous with modulus $k \mapsto k$.
\end{corollary}
\begin{proof}
  Let $n\in \Nat$ and $\alpha,\beta \in \TTree$, and suppose that
  $\overline{\alpha}n = \overline{\beta}n$. Then
  $x^{n}_{\alpha} = x^{n}_{\beta}$, so
   by Lemma \ref{lem:RepresentationApprox}, we have
  \[
    |\Phi(\alpha) - \Phi(\beta) | 
    \leq |x_{\alpha} - x_{\alpha}^{n}| + |x_{\beta} - x_{\beta}^{n}|
    \leq 2^{-(n+1)} + 2^{-(n+1)} = 2^{-n}.
    \qedhere
  \]
\end{proof}

To each node $s \in \TSeq$, assign an interval $\RInt_{s}$ with rational
endpoints (see Figure \ref{fig:Interval}):
\begin{equation}
  \label{eq:AssignmentInterval}
  \RInt_{s} \defeql \left[2^{-(|s| + 1)} (N(s) - 1), 2^{-(|s| + 1)}
  (N(s) + 1) \right].
\end{equation}
Note that the length of $\RInt_{s}$ is $2^{-|s|}$ and the length of
the overlapping
area of 
adjacent intervals $\RInt_{s*\seq{i}}$ and $\RInt_{s*\seq{i+1}}$ is 
$2^{-(|s| + 2)}$.

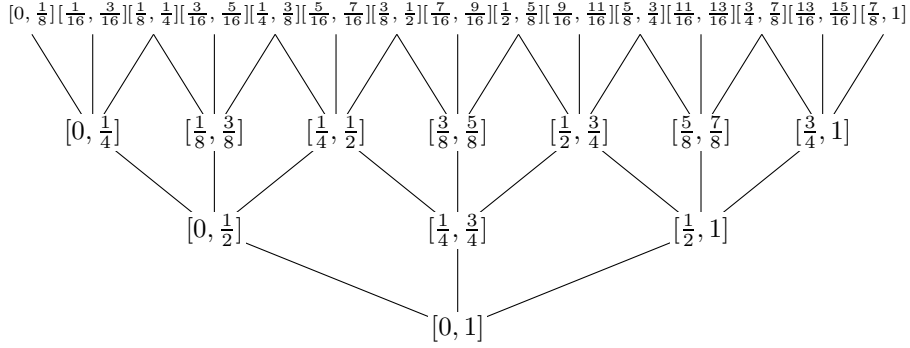
\begin{figure}[tb]
  \centering
\begin{tikzpicture}[xscale=0.8,yscale=1.3]
%  \draw  (-4,0) -- (4,0);
  %
  % 1st level
  %
  \draw (0,0) -- (-4,1);
  \draw (0,0) -- (0,1);
  \draw (0,0) -- (4,1);
  %
  % 2nd level
  %
  \draw (-4,1) -- (-6,2);
  \draw (-4,1) -- (-4,2);
  \draw (-4,1) -- (-2,2);
  \draw (0,1) -- (-2,2);
  \draw (0,1) -- (0,2);
  \draw (0,1) -- (2,2);
  \draw (4,1) -- (2,2);
  \draw (4,1) -- (4,2);
  \draw (4,1) -- (6,2);
  %
  % 3rd level
  %
  \draw (-6,2) -- (-7,3);
  \draw (-6,2) -- (-6,3);
  \draw (-6,2) -- (-5,3);
  \draw (-4,2) -- (-5,3);
  \draw (-4,2) -- (-4,3);
  \draw (-4,2) -- (-3,3);
  \draw (-2,2) -- (-3,3);
  \draw (-2,2) -- (-2,3);
  \draw (-2,2) -- (-1,3);
  \draw (0,2) -- (-1,3);
  \draw (0,2) -- (0,3);
  \draw (0,2) -- (1,3);
  \draw (2,2) -- (1,3);
  \draw (2,2) -- (2,3);
  \draw (2,2) -- (3,3);
  \draw (4,2) -- (3,3);
  \draw (4,2) -- (4,3);
  \draw (4,2) -- (5,3);
  \draw (6,2) -- (5,3);
  \draw (6,2) -- (6,3);
  \draw (6,2) -- (7,3);
  %%%%%%%%%%%%%%%%%%%%
  % Assignment of Intervals
  %%%%%%%%%%%%%%%%%%%%
  %
  % Root
  \draw [] (0,0)  node[inner sep=1pt, fill=white] {$[0,1]$};
  %
  % First level
  %
  \draw [] (-4,1) node[inner sep=1pt, fill=white]
  {$[0,\frac{1}{2}]$};
  \draw [] (0,1)  node[inner sep=1pt, fill=white]
  {$[\frac{1}{4}, \frac{3}{4}]$};
  \draw [] (4,1)  node[inner sep=1pt, fill=white] 
  {$[\frac{1}{2}, 1]$};
  %
  % 2nd level
  %
  \draw [] (-6,2)   node[inner sep=1pt, fill=white]
  {$[0, \frac{1}{4}]$};
  \draw [] (-4,2)   node[inner sep=1pt, fill=white]
  {$[\frac{1}{8}, \frac{3}{8}]$};
  \draw [] (-2,2)   node[inner sep=1pt, fill=white]
  {$[\frac{1}{4}, \frac{1}{2}]$};
  \draw [] (0,2)    node[inner sep=1pt, fill=white]
  {$[\frac{3}{8}, \frac{5}{8}]$};
  \draw [] (2,2)    node[inner sep=1pt, fill=white]
  {$[\frac{1}{2}, \frac{3}{4}]$};
  \draw [] (4,2)    node[inner sep=1pt, fill=white] 
  {$[\frac{5}{8}, \frac{7}{8}]$};
  \draw [] (6,2)    node[inner sep=1pt, fill=white] 
  {$[\frac{3}{4}, 1]$};
  %
  % 3rd level
  %
  \draw [above] (-7,3)   node[ font=\scriptsize] 
  {$[0, \frac{1}{8}]$};
  \draw [above] (-6,3)   node[ font=\scriptsize] 
  {$[\frac{1}{16}, \frac{3}{16}]$};
  \draw [above] (-5,3)   node[ font=\scriptsize] 
  {$[\frac{1}{8}, \frac{1}{4}]$};
  \draw [above] (-4,3)   node[ font=\scriptsize] 
  {$[\frac{3}{16}, \frac{5}{16}]$};
  \draw [above] (-3,3)   node[ font=\scriptsize] 
  {$[\frac{1}{4}, \frac{3}{8}]$};
  \draw [above] (-2,3)   node[ font=\scriptsize] 
  {$[\frac{5}{16}, \frac{7}{16}]$};
  \draw [above] (-1,3)   node[ font=\scriptsize] 
  {$[\frac{3}{8}, \frac{1}{2}]$};
  \draw [above] (0,3)    node[ font=\scriptsize] 
  {$[\frac{7}{16}, \frac{9}{16}]$};
  \draw [above] (1,3)    node[ font=\scriptsize] 
  {$[\frac{1}{2}, \frac{5}{8}]$};
  \draw [above] (2,3)    node[ font=\scriptsize] 
  {$[\frac{9}{16}, \frac{11}{16}]$};
  \draw [above] (3,3)    node[ font=\scriptsize] 
  {$[\frac{5}{8}, \frac{3}{4}]$};
  \draw [above] (4,3)    node[ font=\scriptsize] 
  {$[\frac{11}{16}, \frac{13}{16}]$};
  \draw [above] (5,3)    node[ font=\scriptsize] 
  {$[\frac{3}{4}, \frac{7}{8}]$};
  \draw [above] (6,3)    node[ font=\scriptsize] 
  {$[\frac{13}{16}, \frac{15}{16}]$};
  \draw [above] (7,3)    node[ font=\scriptsize] 
  {$[\frac{7}{8}, 1]$};
\end{tikzpicture}
\caption{The assignment of closed intervals} \label{fig:Interval}
\end{figure}
 
Given a regular sequence $x = \seq{r_{n}}_{n \in \Nat}$ in $\UInt$,
define a sequence $\seq{\RInt^{x}_{n}}_{n \in \Nat}$ of rational
intervals by
\[
  \RInt^{x}_{n} \defeql \left[ \max\{r_{n+3} - 2^{-(n+3)},0\}, 
  \min\{ r_{n+3} + 2^{-(n+3)}, 1\} \right].
\]
For each $n \in \Nat$, the length of $\RInt^{x}_{n}$ is
less than $2^{-(n+2)}$, which is the length of the 
 overlapping area of $\RInt_{s*\seq{i}}$  and $\RInt_{s*\seq{i+1}}$
for some $s \in \TSeq$ and $i \in \Bin$ such that $\lh{s} = n$.
Thus, there exists $t \in \TSeq$ of length $n+1$ such that
$\RInt_{n}^{x} \sqsubseteq \RInt_{t}$.
By primitive recursion, we can thus define a path $\alpha_{x} \in
\TTree$ as follows:
\begin{equation}
    \label{def:FromRegularToPath}
  \begin{aligned}
    \alpha_{x}(0) &\defeql \text{the least $i \in \Ter$ such that
      $\RInt^{x}_{0} \sqsubseteq \RInt_{\seq{i}}$},\\
      \alpha_{x}(n+1) &\defeql \text{the least $i \in \Ter$ such that
        $\RInt^{x}_{n+1} \sqsubseteq \RInt_{\seq{\alpha_{x}(0),
        \dots,\alpha_{x}(n), i}}$}.
      \end{aligned}
    \end{equation}
By induction, one can show that
\begin{equation}
  \label{eq:PathfromTree}
  \RInt_{n}^{x} \sqsubseteq \RInt_{\overline{\alpha_{x}}(n+1)}
\end{equation}
for all $n \in \Nat$.
Note that the mapping $x \mapsto \alpha_{x}$ does not
preserve the equality on
$\Real$, and thus it is not a function on $\UInt$.%
\footnote{For example, consider $\seq{1/2 +
2^{-(n+1)}}_{n \in \Nat}$ and $\seq{1/2 - 2^{-(n+1)}}_{n \in \Nat}$.}

The following proposition states that every real number in $\UInt$ can
be represented by a path in $\TSeq$ via $\Phi$.
\begin{proposition}
  \label{prop:Surject}
  For each real number $x$ in $\UInt$,
  we have $x \simeq \Phi(\alpha_{x})$.
\end{proposition}
\begin{proof}
  Let $x = \seq{r_{n}}_{n \in \Nat}$ be a regular sequence in
  $\UInt$. 
  Fix $n \in \Nat$.  Since $0 \leq x \leq 1$, we have $-2^{-(n+2)}
  \leq r_{n+3} \leq 1 + 2^{-(n+2)}$. This, together with \eqref{eq:PathfromTree}, 
  implies
  \[
    |r_{n+3} - 2^{-(n+2)}N(\overline{\alpha_{x}}(n+1))| \leq 2^{-(n+1)}.
  \]
  Thus 
  \begin{align*}
    |r_{n+1} - \Phi(\alpha_{x})_{n+1}|
    &=
    |r_{n+1} - 2^{-(n+2)}N(\overline{\alpha_{x}}(n+1))| \\
    &\leq |r_{n+1} - r_{n+3} | + | r_{n+3} -
    2^{-(n+2)}N(\overline{\alpha_{x}}(n+1))| \\
    &\leq 2^{-(n+2)} + 2^{-(n+3)} + 2^{-(n+1)}
    < 2^{-n}.
  \end{align*}
  Therefore $x \simeq \Phi(\alpha_{x})$.
\end{proof}
Our next aim is to prove the quotient property of $\Phi$ (see
Proposition \ref{prop:Quotient}).
Let $\rho \colon \Ter^{3} \to \Ter^{3}$ be the function which is an
identity on $\Ter^{3} $ except for the following patterns:
\begin{equation}
  \label{eq:Pattern}
  \begin{aligned}
    \seq{1, 2, 2} &\stackrel{\rho}{\mapsto} \seq{2, 0, 2} &\qquad\qquad
    \seq{1, 0, 0} &\stackrel{\rho}{\mapsto} \seq{0, 2, 0} \\
    \seq{0, 2, 2} &\stackrel{\rho}{\mapsto} \seq{1, 0, 2} &\qquad\qquad
    \seq{2, 0, 0} &\stackrel{\rho}{\mapsto} \seq{1, 2, 0} 
  \end{aligned}
\end{equation}
The function $\rho$ is extended to $\rho \colon \TTree \to \TTree$ by
primitive recursion:
\begin{equation}
  \label{def:rho}
  \rho(\alpha) \defeql \lambda n. (\sigma^{n}_{\alpha})_{0},
\end{equation}
where $\sigma^{n}_{\alpha} \in \Ter^{3}$ is defined by
\begin{align*}
  \sigma^{0}_{\alpha} 
  &\defeql \rho(\alpha_{0}, \alpha_{1}, \alpha_{2}), &
  \sigma^{n+1}_{\alpha} 
  &\defeql \rho((\sigma^{n}_{\alpha})_{1}, \alpha_{n+2}, \alpha_{n + 3}).
\end{align*}

\begin{lemma}
  \label{lem:RhoImpossiblePattern}
  For any  $\alpha \in \TTree$, $n \in \Nat$, and $i \in \left\{
    0,2 \right\}$,
    \[
      \alpha_{n} \neq i \imp
      \forall m \geq n \left( \seq{\rho(\alpha)_{m}, \rho(\alpha)_{m+1}, \rho(\alpha)_{m+2}} 
      \neq \seq{i,i,i} \right).
    \]
\end{lemma}
\begin{proof}
  We give a proof for $i = 0$. The proof for $i = 2$ is similar.

  Suppose that $\alpha_{n} \neq 0$.
  Suppose that
  $\seq{\rho(\alpha)_{m}, \rho(\alpha)_{m+1}, \rho(\alpha)_{m+2}} = \seq{0,0,0}$ 
  for some $m \geq n$.
  We may assume that $m$ is the least number $\geq n$ with this property. 
  Then,
  \[
    \rho(\alpha)_{m} 
    = (\sigma^{m}_{\alpha})_{0} 
    = (\rho((\sigma^{m-1}_{\alpha})_{1}, \alpha_{m+1},
    \alpha_{m+2}))_{0} = 0.
  \]
  (if $m = n = 0$, we put $(\sigma^{m-1}_{\alpha})_{1} = \alpha_{n}$).

  If $\seq{(\sigma^{m-1}_{\alpha})_{1}, \alpha_{m+1}, \alpha_{m+2}}$
  matches some pattern of $\rho$, then we must have 
  $(\sigma^{m-1}_{\alpha})_{1} = 1$ and $\seq{\alpha_{m+1}, \alpha_{m+2}} =
  \seq{0,0}$. Then, $( \sigma^{m}_{\alpha} )_{1} = 2$, so
  \[
    \rho(\alpha)_{m+1} 
    = (\rho((\sigma^{m}_{\alpha})_{1}, \alpha_{m+2},
    \alpha_{m+3}))_{0}
    = (\rho(2, \alpha_{m+2}, \alpha_{m+3}))_{0}
    = 0,
  \]
  which is impossible.

  If $\seq{(\sigma^{m-1}_{\alpha})_{1}, \alpha_{m+1}, \alpha_{m+2}}$
  does not match any pattern of $\rho$, then we must have 
  $(\sigma^{m-1}_{\alpha})_{1} = 0$, $\seq{\alpha_{m+1}, \alpha_{m+2}}
  \neq \seq{2,2}$, and $( \sigma^{m}_{\alpha} )_{1} = \alpha_{m+1}$.
  If $m = n = 0$, then $( \sigma^{m-1}_{\alpha} )_{1} = \alpha_{0} = 0$, a
  contradiction. Thus, we may assume $m > 0$.
  Then,
  \[
    \rho(\alpha)_{m+1} 
    = (\rho((\sigma^{m}_{\alpha})_{1}, \alpha_{m+2},
    \alpha_{m+3}))_{0}
    = (\rho(\alpha_{m+1}, \alpha_{m+2}, \alpha_{m+3}))_{0}
    = 0.
  \]
  By the definition of $\rho$, the possibility of $\alpha_{m+1} = 2$, $\seq{\alpha_{m+1},
  \alpha_{m+2}} = \seq{1,1}$, or
  $\seq{\alpha_{m+1}, \alpha_{m+2}} = \seq{1,2}$ is ruled out.
  Moreover, $\alpha_{m+1} = 0$ implies $\sigma^{m-1}_{\alpha} =
  \seq{j,0,0}$ for some $j \in \Ter$. By the definition
  of $\rho$, this implies 
  $\seq{(\sigma^{m-2}_{\alpha})_{1},\alpha_{m},\alpha_{m+1}}$ does not match
  any pattern of $\rho$. Thus 
  $\seq{(\sigma^{m-2}_{\alpha})_{1},\alpha_{m},\alpha_{m+1}} =
  \seq{0,0,0} = \seq{j,0,0}$,
  and so $\rho(\alpha)_{m-1} =( \sigma^{m-1}_{\alpha})_{0} = 0$. If $m-1 \geq n$,
  this contradicts the leastness of $m$. Thus $m-1 < n$, and so
  $m = n$. Then $\alpha_{n} = 0$, a contradiction.

  Hence, the only possibility is $\seq{\alpha_{m+1}, \alpha_{m+2}} = \seq{1,0}$.
  Since $\rho(\alpha)_{m+1} = 0$, we must have
  $\alpha_{m+3} = 0$ and $(\sigma^{m+1}_{\alpha})_{1} = 2$.  Thus
  \[
    \rho(\alpha)_{m+2} 
    = (\rho((\sigma^{m+1}_{\alpha})_{1}, \alpha_{m+3},
    \alpha_{m+4}))_{0}
    = (\rho(2, 0, \alpha_{m+4}))_{0}
    = 0,
  \]
  which is impossible.
\end{proof}

\begin{corollary}
  \label{cor:RhoImpossiblePattern}
  For any  $\alpha \in \TTree$, $n \in \Nat$, and $i \in \left\{
    0,2 \right\}$, 
    \[
   \seq{\rho(\alpha)_{n}, \rho(\alpha)_{n+1}, \rho(\alpha)_{n+2}} = \seq{i,i,i}
   \imp 
   \overline{\alpha}(n+1) = i^{n+1}.
    \]
\end{corollary}

The following is intuitively obvious.
\begin{lemma}
  \label{lem:Numbering}
  Let $a,b\in \TSeq$ such that $N(a) = N(b)$. Then, for any $n \in
  \Nat$, we have
  \begin{equation}
    \label{eq:Numbering}
    \forall c,d \in \Ter^{n} \forall k \in \Int
    \left( N(c) + k = N(d) \imp N(a * c) + k = N(b* d) 
    \right).
  \end{equation}
\end{lemma}
\begin{proof}
  Fix $a,b \in \TSeq$ such that $N(a) = N(b)$.
  We show \eqref{eq:Numbering} by induction on $n$. 
  The base case ($n = 0$) is trivial.
  For the inductive case ($n = n' + 1$), let $c,d \in \Ter^{n'}$, 
  $i,j \in \Ter$ and $k \in \Int$, and suppose that 
  $N(c*\seq{i}) + k = N(d*\seq{j})$. Then, 
    $2N(c) + (i - 1) + k = 2N(d) + (j - 1)$. Thus
    $
    N(c) + \frac{(i-j) + k}{2} = N(d),
    $
  where $\left((i-j) + k  \right) / 2$ is an integer. By induction
  hypothesis, we have
    $
    N(a*c) + \frac{(i-j) + k}{2} = N(b*d).
    $
  Hence
  \begin{align*}
    N(a*c*\seq{i}) + k 
    &= 
    2N(a*c) + (i - 1) + k \\
    &= 
    2N(a*c) + (i - j) + k + (j - 1)  \\
    &= 
    2N(b*d) + (j - 1)  \\
    &= 
    N(b*d*\seq{j}).
    \qedhere
  \end{align*}
\end{proof}

\begin{lemma}
  \label{lem:ReductionTTree}
  For each $\alpha \in \TTree$ and $n \in \Nat$, we have
  \begin{enumerate}
    \item\label{lem:ReductionTTree1}
      $
      (\sigma^{n}_{\alpha})_{1}
      =
      \begin{cases}
        \alpha_{n+1} & \text{if $N(\overline{\rho(\alpha)}(n+1)) =
        N(\overline{\alpha}(n+1))$}, \\
        2 - \alpha_{n+1} & \text{otherwise}.
      \end{cases}
      $
    \item\label{lem:ReductionTTree2}
      $N(\overline{\alpha}(n+1)) \neq
      N(\overline{\rho(\alpha)}(n+1)) \imp N(\overline{\alpha}(n+2)) = N(\overline{\rho(\alpha)}(n+1) *
      \seq{2-\alpha_{n+1}})$.
  \end{enumerate}
\end{lemma}
\begin{proof}
  We show \ref{lem:ReductionTTree1}
  and \ref{lem:ReductionTTree2} by simultaneous induction.
  \smallskip

  \noindent\emph{Base case $(n = 0)$}:
  For \ref{lem:ReductionTTree1}, if 
  $N(\overline{\rho(\alpha)}1) = N(\overline{\alpha}1)$, then
  $\rho(\alpha)_{0} = \alpha_{0}$. This means that
  $\seq{\alpha_{0},\alpha_{1},\alpha_{2}}$ does not match any
  pattern in \eqref{eq:Pattern}.
  Thus, $(\sigma^{0}_{\alpha})_{1} = \alpha_{1}$. 
  If $N(\overline{\rho(\alpha)}1) \neq N(\overline{\alpha}1)$, then 
 $\seq{\alpha_{0},\alpha_{1},\alpha_{2}}$ matches some 
  pattern in \eqref{eq:Pattern}, which implies 
  $(\sigma^{0}_{\alpha})_{1} = 2 - \alpha_{1}$. 
 The base case for \ref{lem:ReductionTTree2} can be proved similarly.
  \smallskip

  \noindent\emph{Inductive case $(n = k + 1)$}:
  Assume \ref{lem:ReductionTTree1}
  and \ref{lem:ReductionTTree2} for $k$.
  First, we show \ref{lem:ReductionTTree1}. Suppose
  \begin{equation}
    \label{lem:ReductionTTreeInd}
    N(\overline{\rho(\alpha)}(k+2)) = N(\overline{\alpha}(k+2)).
  \end{equation}

  \noindent\emph{Case $N(\overline{\rho(\alpha)}(k+1)) =
  N(\overline{\alpha}(k+1))$}: By the induction hypothesis of 
  \ref{lem:ReductionTTree1}, we have
  $(\sigma^{k}_{\alpha})_{1} = \alpha_{k+1}$.
  On the other hand, by \eqref{lem:ReductionTTreeInd}
  and the assumption $N(\overline{\rho(\alpha)}(k+1)) =
  N(\overline{\alpha}(k+1))$, we have
    $\rho(\alpha)_{k+1} = \alpha_{k+1}$.
  Thus
  \[
    \alpha_{k+1}
    = \rho(\alpha)_{k+1} 
    = (\sigma^{k+1}_{\alpha})_{0}
    = (\rho((\sigma^{k}_{\alpha})_{1}, \alpha_{k+2},
    \alpha_{k+3}))_{0}
    = (\rho(\alpha_{k+1}, \alpha_{k+2}, \alpha_{k+3}))_{0},
  \]
  which means that $\seq{\alpha_{k+1}, \alpha_{k+2}, \alpha_{k+3}}$ does not match
  any pattern in \eqref{eq:Pattern}. Thus, we must have
  $(\sigma^{k+1}_{\alpha})_{1} =  \alpha_{k+2}$.
  \smallskip

  \noindent\emph{Case $N(\overline{\rho(\alpha)}(k+1)) \neq
  N(\overline{\alpha}(k+1))$}: By the induction hypothesis of 
\ref{lem:ReductionTTree2}, we have
\begin{equation}
  \label{lem:ReductionTTreeInd2}
N(\overline{\alpha}(k+2)) = N(\overline{\rho(\alpha)}(k+1) *
      \seq{2-\alpha_{k+1}}).
\end{equation}
By the induction hypothesis of 
\ref{lem:ReductionTTree1}, we also have
$(\sigma^{k}_{\alpha})_{1} = 2 - \alpha_{k+1}$.
By \eqref{lem:ReductionTTreeInd} and \eqref{lem:ReductionTTreeInd2}, we have
  $
  \rho(\alpha)_{k+1}
  = 2 - \alpha_{k + 1}.
  $
Thus
  \begin{align*}
    2 - \alpha_{k+1}
    = \rho(\alpha)_{k+1} 
    = (\rho((\sigma^{k}_{\alpha})_{1}, \alpha_{k+2},
    \alpha_{k+3}))_{0}
    = (\rho(2- \alpha_{k+1}, \alpha_{k+2}, \alpha_{k+3}))_{0},
  \end{align*}
which means that
$\seq{2 - \alpha_{k + 1}, \alpha_{k+2}, \alpha_{k + 3}}$ does not match any
pattern in \eqref{eq:Pattern}. Thus
  $
  (\sigma^{k+1}_{\alpha})_{1} = \alpha_{k+2}
  $.
  \smallskip

  Next, contrary to \eqref{lem:ReductionTTreeInd}, suppose that
  \begin{equation}
    \label{lem:ReductionTTreeInd3}
    N(\overline{\rho(\alpha)}(k+2)) \neq N(\overline{\alpha}(k+2)).
  \end{equation}
  \smallskip
  \noindent\emph{Case $N(\overline{\rho(\alpha)}(k+1)) =
  N(\overline{\alpha}(k+1))$}: By the induction hypothesis of 
  \ref{lem:ReductionTTree1}, we have
  $(\sigma^{k}_{\alpha})_{1} = \alpha_{k+1}$. By
  \eqref{lem:ReductionTTreeInd3} and the assumption $N(\overline{\rho(\alpha)}(k+1)) =
  N(\overline{\alpha}(k+1))$, we have 
  $\rho(\alpha)_{k+1} \neq \alpha_{k+1}$. Thus
  \begin{align*}
    \alpha_{k+1}
    \neq \rho(\alpha)_{k+1}
    = (\rho((\sigma^{k}_{\alpha})_{1},
    \alpha_{k+2},\alpha_{k+3}))_{0} 
    = (\rho(\alpha_{k+1}, \alpha_{k+2},\alpha_{k+3}))_{0},
  \end{align*}
  which means that $\seq{\alpha_{k+1}, \alpha_{k+2}, \alpha_{k+3}}$
  matches some pattern in $\eqref{eq:Pattern}$. By the inspection of
  the patterns in $\eqref{eq:Pattern}$, we must have
  $(\sigma^{k+1}_{\alpha})_{1} = 2 - \alpha_{k+2}$.
  \smallskip

  \noindent\emph{Case $N(\overline{\rho(\alpha)}(k+1)) \neq
  N(\overline{\alpha}(k+1))$}: By the induction hypothesis of 
  \ref{lem:ReductionTTree2}, we have
  $N(\overline{\alpha}(k+2)) = N(\overline{\rho(\alpha)}(k+1) *
  \seq{2 - \alpha_{k+1}})$, and so by 
  \eqref{lem:ReductionTTreeInd3}, we must have $\rho(\alpha)_{k+1} \neq 2 -
  \alpha_{k+1}$. 
  On the other hand, by the induction hypothesis of
\ref{lem:ReductionTTree1}, we have $(\sigma^{k}_{\alpha})_{1}= 2 -
\alpha_{k+1}$. Thus
  \[
   2 - \alpha_{k+1}
   \neq \rho(\alpha)_{k+1}
    = (\rho((\sigma^{k}_{\alpha})_{1}, \alpha_{k+2},\alpha_{k+3}))_{0}
    = (\rho(2 - \alpha_{k+1}, \alpha_{k+2},\alpha_{k+3}))_{0},
  \]
 which means that $\seq{2 - \alpha_{k+1},
  \alpha_{k+2}, \alpha_{k+ 3}}$ matches some pattern in
  \eqref{eq:Pattern}. 
  By the inspection of
  the patterns in $\eqref{eq:Pattern}$, we must have
  $(\sigma^{k+1}_{\alpha})_{1} = 2 -
  \alpha_{k+2}$.
  \smallskip

  Next, we show \ref{lem:ReductionTTree2}. Suppose
  \begin{equation}
    \label{lem:ReductionTTreeInd4}
    N(\overline{\rho(\alpha)}(k+2)) \neq N(\overline{\alpha}(k+2)).
  \end{equation}

  \noindent\emph{Case $N(\overline{\rho(\alpha)}(k+1)) =
  N(\overline{\alpha}(k+1))$}: By \eqref{lem:ReductionTTreeInd4} and
  the assumption $N(\overline{\rho(\alpha)}(k+1)) =
  N(\overline{\alpha}(k+1))$, we have $\rho(\alpha)_{k+1} \neq
  \alpha_{k+1}$. On the other hand, by the induction hypothesis of
  $\ref{lem:ReductionTTree1}$, we have
  $(\sigma^{k}_{\alpha})_{1} = \alpha_{k+1}$. Thus
  \begin{align*}
    \alpha_{k+1}
    \neq \rho(\alpha)_{k+1}
    = (\rho( (\sigma^{k}_{\alpha})_{1}, \alpha_{k+2},
    \alpha_{k+3}))_{0} 
    = (\rho(\alpha_{k+1}, \alpha_{k+2}, \alpha_{k+3}))_{0},
  \end{align*}
  which means that $\seq{\alpha_{k+1}, \alpha_{k+2}, \alpha_{k+3}}$ matches some
  pattern in \eqref{eq:Pattern}. By the inspection of the patterns in 
  \eqref{eq:Pattern}, we have 
  \[
    N(\seq{\alpha_{k+1}, \alpha_{k+2}}) =N(\seq{\rho(\alpha)_{k+1}, 2
    - \alpha_{k+2}}).
  \]
  and hence, by  Lemma \ref{lem:Numbering}, we obtain
  \[
    N(\overline{\alpha}(k+3)) 
    = N(\overline{\rho(\alpha)}(k+1) * \seq{\alpha_{k+1}, \alpha_{k+2}}) 
    = N(\overline{\rho(\alpha)}(k+2) * \seq{ 2 - \alpha_{k+2}}).
  \]

  \noindent\emph{Case $N(\overline{\rho(\alpha)}(k+1)) \neq
  N(\overline{\alpha}(k+1))$}: By the induction hypothesis of
  $\ref{lem:ReductionTTree2}$, we have
    $
    N(\overline{\alpha}(k+2))
    = N(\overline{\rho(\alpha)}(k+1) * \seq{2 - \alpha_{k+1}}),
    $
  and so  $\rho(\alpha)_{k+1} \neq 2- \alpha_{k+1}$
  by \eqref{lem:ReductionTTreeInd4}.
  On the other hand, by the induction hypothesis of
  $\ref{lem:ReductionTTree1}$,
  we have $(\sigma^{k}_{\alpha})_{1} = 2 - \alpha_{k+1}$. Thus
  \[
    2- \alpha_{k+1}
    \neq \rho(\alpha)_{k+1}
    = (\rho( (\sigma^{k}_{\alpha})_{1}, \alpha_{k+2}, \alpha_{k+3}))_{0}
    = (\rho(2- \alpha_{k+1}, \alpha_{k+2}, \alpha_{k+3}))_{0},
  \]
  which means that
  $\seq{2 - \alpha_{k+1}, \alpha_{k+2}, \alpha_{k+3}}$ matches some pattern
  in \eqref{eq:Pattern}. Then, by the similar argument as in the
  previous case, we have
  \[
    N(\overline{\alpha}(k+3))
    = N(\overline{\rho(\alpha)}(k+1) * \seq{2 - \alpha_{k+1},
    \alpha_{k+2}})
    = N(\overline{\rho(\alpha)}(k+2) * \seq{2 - \alpha_{k+2}}).
    \qedhere
  \]
\end{proof}
\begin{corollary}
  \label{cor:ReductionTree}
  For any $\alpha \in \TTree$ and $n \in \Nat$, we have
  \[
    |N(\overline{\alpha}(n+1)) - N(\overline{\rho(\alpha)}(n+1))| \leq 1.
  \]
\end{corollary}
\begin{proof}
  If $N(\overline{\alpha}(n+1)) = N(\overline{\rho(\alpha)}(n+1))$,
  the conclusion is immediate. Suppose that
  $N(\overline{\alpha}(n+1)) \neq N(\overline{\rho(\alpha)}(n+1))$.
  By Lemma \ref{lem:ReductionTTree}, we have
  $N(\overline{\alpha}(n+2)) = N(\overline{\rho(\alpha)}(n+1) *
  \seq{2 - \alpha_{n+1}})$, which implies
  $|N(\overline{\alpha}(n+1)) - N(\overline{\rho(\alpha)}(n+1)) | = 1$
  (see Figure \ref{fig:Numbering}).
\end{proof}

\begin{corollary}
  \label{cor:rhoPreserveEqual}
  For any  $\alpha \in \TTree$, we have
    $
    \Phi(\alpha) \simeq \Phi(\rho(\alpha)).
    $
\end{corollary}
\begin{proof}
  Immediate from Corollary \ref{cor:ReductionTree}.
\end{proof}

The following proposition, together with Proposition \ref{prop:Surject}
and Corollary \ref{cor:rhoPreserveEqual},
 states that $\UInt$ is a \emph{uniform quotient}~\cite{VarietiesConstMath} of
 $\TTree$ (cf.\ Troelstra and van Dalen \cite[Chapter 6, Proposition
 3.2 (iii)]{ConstMathI}).
\begin{proposition}
  \label{prop:Quotient}
  For each $\alpha \in \TTree$ and $n \in \Nat$, 
  \[
    \forall x \in \UInt \left( |x - \Phi(\rho(\alpha))| <
    2^{-(n+5)} \imp \exists \gamma \in \overline{\rho(\alpha)} n
    \left(  x
    \simeq \Phi(\gamma)\right)\right).
  \]
\end{proposition}
\begin{proof}
  Fix $\alpha \in \TTree$ and $n \in \Nat$.
  By Proposition \ref{prop:Surject}, it suffices to show that
  \[
     | \Phi(\beta) - \Phi(\rho(\alpha))| < 2^{-(n+5)} 
     \imp
     \exists \gamma \in \overline{\rho(\alpha)} n
     \left(  \Phi(\beta) \simeq \Phi(\gamma)\right)
  \]
  for all $\beta \in \TTree$.
  Fix $\beta \in \TTree$, and suppose that
  $|\Phi(\beta) - \Phi(\rho(\alpha))| < 2^{-(n+5)}$.
  Then, $|x_{\beta}^{m} - x_{\rho(\alpha)}^{m}| < 2^{-(n+5)}$ for
  sufficiently large $m \geq n+4$. Thus
  \begin{align*}
    &|2^{-(n+5)} N(\overline{\beta}(n+4))
    - 2^{-(n+5)} N(\overline{\rho(\alpha)}(n+4))| \\
  &\leq |2^{-(n+5)} N(\overline{\beta}(n+4)) -  x_{\beta}^{m}|
  + | x_{\beta}^{m} - x_{\rho(\alpha)}^{m}| 
  + | x_{\rho(\alpha)}^{m} - 2^{-(n+5)} N(\overline{\rho(\alpha)}(n+4))| \\
  &< 2^{-(n+5)} + 2^{-(n+5)} + 2^{-(n+5)}\\
  &= 3 \cdot 2^{-(n+5)}.
  \end{align*}
  Hence $| N(\overline{\beta}(n+4)) - N(\overline{\rho(\alpha)}(n+4)) |
  \leq 2$. Since $\seq{\rho(\alpha)_{n}, \rho(\alpha)_{n+1},
  \rho(\alpha)_{n+2}} \notin \left\{ \seq{0,0,0}, \seq{2,2,2} \right\}$ 
  unless $\overline{\rho(\alpha)}(n+3)$ is
  the left most or the right most node of $\Ter^{*}$ (see Corollary
  \ref{cor:RhoImpossiblePattern}),
  we must have
  \[
    |2^{4} N(\overline{\rho(\alpha)}n) - N(\overline{\beta}(n+4))|
    \leq 2^{4} - 1
  \]
  (see Figure \ref{fig:Numbering}). Thus, there exists $s \in
  \Ter^{4}$ such that 
    $
    N(\overline{\rho(\alpha)}n * s) = N(\overline{\beta}(n+4)).
    $
  Hence, the sequence $\gamma \defeql \overline{\rho(\alpha)}n * s *
  \lambda k. \beta(n + 4 + k)$ satisfies $\Phi(\beta) \simeq \Phi(\gamma)$.
\end{proof}

\begin{definition}
  A function $f \colon \UInt\to \Real$ is \emph{uniformly continuous} if
  there exists $\omega \colon \Nat \to \Nat$, called a \emph{modulus
  of uniform continuity}, such that
  \[
    \forall k \in \Nat \forall x, y \in
    \UInt
    \left( |x-y| \leq 2^{-\omega(k)} \imp |f(x) - f(y) | \leq 2^{-k}\right).
  \]
\end{definition}

The following theorem states that the uniform structure of $\UInt$ is
completely determined by $\TTree$ through $\Phi$.
\begin{theorem}
  \label{prop:UniversalQuotient}
  A function $f \colon \UInt \to \Real$ is uniformly continuous if and
  only if the composition $f \circ \Phi \colon \TTree \to \Real$ is
  uniformly continuous.
\end{theorem}
\begin{proof}
  It suffices to show ``if'' part. Suppose that $f \circ \Phi$ is
  uniformly continuous with modulus $\omega \colon \Nat \to \Nat$.
  Fix $k \in \Nat$, and let $x,y$ be regular sequences in $\UInt$
  such that $|x-y| \leq 2^{-(\omega(k)+6)}$. 
  Let $\alpha_{x} \in \TTree$ be the path 
  determined by $x$ by \eqref{def:FromRegularToPath}. Then $x
  \simeq \Phi(\alpha_{x}) \simeq
  \Phi(\rho(\alpha_{x}))$ by Proposition~\ref{prop:Surject}
  and Corollary~\ref{cor:rhoPreserveEqual}.
  Thus, there exists $\beta \in
  \overline{\rho(\alpha_{x})}\omega(k)$ such that $y \simeq \Phi(\beta)$ 
  by Proposition~\ref{prop:Quotient}.
  Then
  \[
    \lh{ f(x) - f(y) }
    \simeq \lh{ f(\Phi(\rho(\alpha_{x}))) - f(\Phi(\beta)) } \leq 2^{-k}.
  \]
  Therefore $f$ is uniformly continuous with modulus
  $k \mapsto \omega(k)+6$.
\end{proof}

\section{Uniform continuity theorem with continuous modulus}
\label{sec:UCTc}
We introduce a notion of modulus for functions from $\UInt$ to
$\Real$ and show that the uniform continuity theorem for
the functions from $\UInt$ to $\Real$ with continuous modulus
($\UCTc$) is equivalent to the decidable fan theorem.

\subsection{Continuous moduli of functions from $\UInt$ to $\Real$}
\label{sec:ContinuousModulus}
We fix a bijective coding of rational numbers by $\Nat$. Let
$\Reg$ denote the set of regular sequences of rational numbers in
$\UInt$,
which is identified with a subset of $\Baire$ through the fixed coding.
Note that the equality on $\Reg$ is the pointwise equality and not 
that of real numbers defined by \eqref{def:RegularEqual}.
In the following, we assume that real numbers are represented by
regular sequences.
\begin{definition}
  \label{def:ModonReg}
  A function $g \colon \Nat \to \Reg \to \Nat$ is a 
  \emph{modulus} of a function $f \colon \UInt \to \Real$ if
  \[
    \forall k \in \Nat \forall x, y \in \UInt
    \left( |x -  y| \leq 2^{-g_{k}(x)} \imp |f(x) - f(y)| \leq 2^{-k} \right).
  \]
  A modulus $g$ of $f \colon \UInt \to \Real$ is \emph{continuous}
  if for each $k \in \Nat$, the function $g_{k} \colon \Reg \to \Nat$
  is pointwise continuous in the sense that
  \[
    \forall x \in \Reg \exists n \in \Nat
    \forall y \in \Reg \left( \overline{x}n = \overline{y}n \imp
    g_{k}(x) = g_{k}(y)\right).
  \]
\end{definition}

We also introduce another notion of modulus for functions from $\UInt$ to
$\Real$, which is defined in terms of the spread representation. 
\begin{definition}
  \label{def:ModonTTree}
  A function $g \colon \Nat \to \TTree \to \Nat$ is a 
  \emph{ternary modulus} of a function $f \colon \UInt \to \Real$ if
  \[
    \forall k \in \Nat \forall \alpha \in \TTree \forall x \in \UInt
    \left( |\Phi(\alpha) - x| \leq 2^{-g_{k}(\alpha)} \imp
    |f(\Phi(\alpha)) -
    f(x)| \leq 2^{-k} \right).
  \]
  A ternary modulus $g$ of $f \colon \UInt \to \Real$ is \emph{continuous}
  if $g_{k} \colon \TTree \to \Nat$ is pointwise continuous for each
  $k \in \Nat$; we say that $g$ is \emph{uniformly continuous}
  if $g_{k} \colon \TTree \to \Nat$ is uniformly continuous for each
  $k \in \Nat$.
  Here, the notion of continuity on $\TTree$ is analogous
  to the one given in Definition \ref{def:ContonCantor}.
\end{definition}

\begin{proposition}
  \label{prop:EquivModTernMod}
  A function $f \colon \UInt \to \Real$ has a continuous modulus if
  and only if $f$ has a continuous ternary modulus.
\end{proposition}
\begin{proof}
  Suppose that $f$ has a continuous modulus $g \colon \Nat \to
  \Reg \to \Nat$.
  Then, the function  $h \colon \Nat \to
  \TTree \to \Nat$ defined by
  \[
    h_{k}(\alpha) \defeql g_{k}(\Phi(\alpha))
  \]
  is a continuous ternary modulus of $f$.

  Conversely, suppose that $f$ has a continuous ternary modulus $g \colon \Nat \to
  \TTree \to \Nat$.  Define $h \colon \Nat \to \Reg \to \Nat$ by 
  \[
    h_{k}(x) \defeql g_{k}(\alpha_{x}),
  \]
  where $\alpha_{x} \in \TTree$ is the path determined by the regular
  sequence $x$ by
  \eqref{def:FromRegularToPath}.
  Since the value of $\alpha_{x}$ at index $n \in \Nat$ depends
  only on the first $n+3$ terms of $x$, the function
  $h_{k}$ is pointwise continuous for each $k \in \Nat$. Fix
  $k \in \Nat$ and $x, y \in \UInt$, and suppose that 
  $|x - y| \leq 2^{-h_{k}(x)}$. By Proposition \ref{prop:Surject}, we have
  \[
    |\Phi(\alpha_{x}) - y| \simeq |x - y| \leq 2^{-h_{k}(x)} =
    2^{-g_{k}(\alpha_{x})}.
  \]
  Since $g$ is a ternary  modulus of $f$, we have $|f(x)-f(y)| \simeq
  |f(\Phi(\alpha_{x}))-f(y)| \leq 2^{-k}$. 
  Therefore $h$ is a continuous modulus of $f$.
\end{proof}

\begin{definition}
The \emph{uniform continuity theorem with continuous modulus}
($\UCTc$) is the following statement:
\begin{description}
\item[($\UCTc$)] Every continuous function $f \colon \UInt \to \Real$ with
  a continuous modulus is uniformly continuous.
\end{description}
We also introduce the following variant of $\UCTc$ formulated with respect to the notion of ternary modulus:
\begin{description}
\item[($\UCTcT$)] Every continuous function $f \colon \UInt \to \Real$ with
  a continuous ternary modulus is uniformly continuous.
\end{description}
\end{definition}

\begin{theorem}
  \label{cor:EqiuvUCTcUTCcT}
  $\UCTc$ and $\UCTcT$ are equivalent.
\end{theorem}
\begin{proof}
  Immediate from Proposition \ref{prop:EquivModTernMod}.
\end{proof}

\subsection{Fan theorem}
We recall some basic notions related to the fan theorem; see Troelstra
and van Dalen~\cite[Chapter 4, Section 7]{ConstMathI} for details.
\begin{definition}
  A subset $B \subseteq \BSeq$ is a \emph{bar} if
  \begin{equation}
    \label{eq:Bar}
    \forall \alpha \in \Cantor \exists n \in \Nat B(\overline{\alpha}n).
  \end{equation}
  A bar $B$ is \emph{uniform} if
  \begin{equation}
    \label{eq:UniformBar}
    \exists N \in \Nat \forall \alpha \in \Cantor \exists n \leq N B(\overline{\alpha}n).
  \end{equation}
  The \emph{decidable fan theorem} reads:
  \begin{description}
    \item[($\DFAN$)] Every decidable bar is uniform.
  \end{description}
\end{definition}
Let $\Nat^{*}$ be the set of finite sequences of $\Nat$.
A \emph{fan} is a 
decidable subset $T \subseteq \FSeq$ such that 
\begin{enumerate}
  \item  $\nil \in T$,
  %\item
  %$\forall s, s' \left(s\in T \land s'\preceq s \to s' \in T \right)$,
  \item  $\forall s \in \FSeq 
    \left(s \in T \leftrightarrow \exists n \in \Nat \left( s *
    \seq{n} \in T \right)  \right)$,
  \item $\exists \beta \in \Nat^{\Nat^{*}} \forall s\in T  
          \forall n \in \Nat \left( s * \seq{n} \in T \imp n \leq \beta(s) \right)$.
\end{enumerate}
A sequence $\alpha \colon \Nat \to \Nat$ is a \emph{path} in $T$, written
$\alpha \in T$, if $\forall n
\in \Nat \left( \overline{\alpha}n \in T \right)$. 
For a  fan $T$, the notion of bar and that of uniform bar are defined
as subsets of $T$ satisfying the conditions analogous to
\eqref{eq:Bar} and \eqref{eq:UniformBar}, where each occurrence of
$\forall \alpha \in \Cantor$ is replaced with $\forall \alpha \in T$.
Then one can generalise $\DFAN$ as  
\begin{description}
  \item[($\DFAN_T$)] Every decidable bar of a  fan $T$ is uniform,
\end{description}
with fan $T$ being a parameter.
Troelstra and van Dalen~\cite[Chapter~4, Proposition~7.5]{ConstMathI} show that for any fan $T$, $\DFAN$ derives $\DFAN_T$.
In particular, $\DFAN$ derives $\DFAN_{\TSeq}$.
On the other hand, since $\BSeq$ is a subfan of $\TSeq$, the proof of
Troelstra and van Dalen~\cite[Chapter~4, Proposition~7.5]{ConstMathI}
shows that $\DFAN_{\TSeq}$ derives $\DFAN$ as well.

In the same manner,
we can consider
the following variations of $\UCc$ and $\UCTcC$ for $\TSeq$:
\begin{description}
  \item[($\UCc_{\TSeq}$)] Every continuous function $f \colon \TSeq \to
  \Nat$ with a continuous modulus is uniformly continuous.
  \item[($\UCTc_{\TSeq}$)] Every continuous function $f \colon \TSeq \to
  \Real$ with a continuous modulus is uniformly continuous.
\end{description}
Here, a \emph{modulus} of $f \colon \TSeq \to
  \Nat$ (or $f \colon \TSeq \to \Real$) is a function $g \colon \TSeq \to \Nat$ (or $g \colon \Nat \to \TSeq \to \Nat$)
satisfying the condition analogous to
\eqref{eq:Modulus} (or \eqref{eq:ModulusRealValued}),
where each occurrence of $\Cantor$ is replaced by $\TTree$.
The proof of the equivalence between $\DFAN$, $\UCc$, and $\UCTcC$ carries over to $\TSeq$.
Thus,
the principles $\DFAN_{\TSeq}$, $\UCc_{\TSeq}$, and $\UCTc_{\TSeq}$ are pairwise equivalent.
Hence we have the following.
\begin{proposition}
  \label{eq:EquivalenceDFANandDFANTernary}
The principles $\DFAN$, $\UCc$, $\UCTcC$,
$\DFAN_{\TSeq}$, $\UCc_{\TSeq}$, and $\UCTc_{\TSeq}$ are pairwise equivalent.
\end{proposition}

\subsection{Equivalence of $\DFAN$ and $\UCTc$}
\label{sec:EquivalenceDFANUCTc}
First, we show that $\DFAN$ implies $\UCTcT$ with a help of the
following lemma.
\begin{lemma}
  \label{lem:UCMod}
  If $f \colon \UInt \to \Real$ has a uniformly continuous ternary modulus,
  then $f$ is uniformly continuous.
\end{lemma}
\begin{proof}
  Let $g \colon \Nat \to \TTree \to \Nat$ be a ternary modulus of $f$, where
  $g_{k}$ is uniformly continuous for each $k \in \Nat$. Fix
  $k \in \Nat$, and let $N_{k}$ be the least modulus of uniform continuity
  of $g_{k}$ (cf.\ Remark \ref{rem:UnifContOnCantor}). Put
  \[
    \omega(k) \defeql \max \left\{ g_{k}(\widehat{s}) \mid s \in
    \Ter^{N_{k}} \right\}.
  \]
  Let $x, y \in \UInt$ be regular sequences such that $|x - y| \leq
  2^{- \omega(k)}$.  By Proposition~\ref{prop:Surject}, we 
  have $x \simeq \Phi(\alpha_{x})$.
  Since $g_{k}(\alpha_{x}) \leq \omega(k)$, we have
  $|\Phi(\alpha_{x}) - y |
  \leq 2^{-g_{k}(\alpha_{x})}$.  Since $g$ is a ternary modulus of $f$,
    \[
      |f(x) - f(y)| \simeq |f(\Phi(\alpha_{x})) - f(y)| \leq 2^{-k}.
    \]
  Therefore $f$ is uniformly continuous with modulus $\omega$.
\end{proof}

\begin{proposition}
  \label{prop:DFANimpliesUContMod}
  $\DFAN$ implies $\UCTcT$.
\end{proposition}
\begin{proof}
  Assume $\DFAN$.
  Let $f \colon \UInt \to \Real$ be a function with a continuous
  ternary modulus $g \colon \Nat \to \TTree \to \Nat$. Define a function 
  $G \colon \Nat \to \TTree \to \Nat$ by 
  \[
    G_{k}(\alpha) \defeql \text{the least $n \in \Nat$ such that
      $g_{k + 1}(\widehat{\overline{\alpha}n}) < n$}.
  \]
  For each $k \in \Nat$, the function $G_{k}$
  is clearly a continuous modulus of itself.  We show that $G$ is a ternary modulus of $f$.
  Let $k \in \Nat$, $\alpha \in \TTree$, and 
  $x \in \UInt$, and suppose that $|\Phi(\alpha)- x| \leq
    2^{- G_{k}(\alpha)}$. By Corollary \ref{cor:UnifContPhi}, we
    have
    \[
      |\Phi(\widehat{\overline{\alpha}G_{k}(\alpha)}) - \Phi(\alpha)|
      \leq 2^{-G_{k}(\alpha)}
      \leq 2^{-g_{k + 1}(\widehat{\overline{\alpha}G_{k}(\alpha)})}
    \]
    and
    \[
      \begin{aligned}
        |\Phi(\widehat{\overline{\alpha}G_{k}(\alpha)}) - x| 
        &\leq
        |\Phi(\widehat{\overline{\alpha}G_{k}(\alpha)}) - \Phi(\alpha)| +
        |\Phi(\alpha) - x|  \\
        &\leq
        2^{-G_{k}(\alpha)} + 2^{-G_{k}(\alpha)}\\
        &\leq
        2^{-g_{k + 1}(\widehat{\overline{\alpha}G_{k}(\alpha)})}.
      \end{aligned}
    \]
  Since $g$ is a ternary modulus of $f$, 
  \[
    \begin{aligned}
      |f(\Phi(\alpha)) - f(x)| 
      &\leq 
      |f(\Phi(\alpha)) - f(\Phi(\widehat{\overline{\alpha}G_{k}(\alpha)}))|
      + |f(\Phi(\widehat{\overline{\alpha}G_{k}(\alpha)})) - f(x)| \\
      &\leq
      2^{-(k+1)} + 2^{-(k+1)}
      =
      2^{-k}.
    \end{aligned}
  \]
  Hence $G$ is a ternary modulus of $f$. 
  Now $G_{k}$ is uniformly
  continuous by Proposition~\ref{eq:EquivalenceDFANandDFANTernary}.
  Therefore $f$ is uniformly continuous by
  Lemma~\ref{lem:UCMod}.
\end{proof}

To show that $\UCTcT$ implies $\DFAN$, we construct from a decidable
bar $B \subseteq \BSeq$ a function $f \colon \UInt \to \Real$
with a continuous ternary modulus in such a way that uniform continuity of
$f$ implies uniformity of $B$. The construction of 
$f$ from $B$ is analogous to those of Loeb \cite[Theorem
5.1]{Loeb05} and Bridges and Diener~\cite{PseudoCompUIntEquivUCT}, but we also need to construct a continuous ternary modulus of
$f$. The reader should consult Notation \ref{not:general} and
Notation \ref{not:Interval}.

The \emph{Cantor's discontinuum} is the image of the function $\kappa
\colon \Cantor \to \UInt$ defined by
\begin{equation}
  \label{def:kappa}
  \kappa(\alpha) 
  \defeql 
  \Bigl\langle \sum_{i < n}2 \alpha_i 3^{-(i+1)}\Bigr\rangle_{n \in \Nat}.
\end{equation}
To each $s \in \BSeq$, assign an interval $\CSet_s$ with
rational endpoints:
\[
  \CSet_s
  \defeql
  \left[ \sum_{i < \lh{s}} 2 s_{i} 3^{-(i+1)},\;
   3^{-\lh{s}} + \sum_{i < \lh{s}} 2 s_{i} 3^{-(i+1)}  \right].
\]
For each $n \in \Nat$ and $s \in \Bin^{n}$, the interval $\CSet_s$
is in the $n$-th level of Cantor's middle-third sets, which is of
length $3^{-n}$.

Let $L \colon \Nat \to \Nat$ be the function defined by 
\[
  L(k) \defeql \text{the least $n$ such that $2^{-n} \leq 3^{-k}$}.
\]
To each $\alpha \in \TTree$, assign a binary sequence $\gamma_{\alpha}
\in \Cantor$ by primitive recursion:
\[
  \begin{aligned}
    \gamma_{\alpha}(0) 
    &\defeql 
    \begin{cases}
      0 & \text{if $\snd{\RInt_{\overline{\alpha}L(1)}} <
      \fst{\CSet_{\seq{1}}}$},\\
      1 & \text{otherwise},
    \end{cases} \\
    \gamma_{\alpha}(n+1) 
    &\defeql 
    \begin{cases}
      0 & \text{if $\snd{\RInt_{\overline{\alpha}L(n+2)}} <
      \fst{\CSet_{\seq{\gamma_{\alpha}(0), \dots,\gamma_{\alpha}(n),
    1}}}$},\\
      1 & \text{otherwise},
    \end{cases}
  \end{aligned}
\]
where $s \mapsto \RInt_s$ is defined as in
\eqref{eq:AssignmentInterval}.

For two paths $\alpha,\beta \in \TTree$ which represent a same real
number (i.e., $\Phi(\alpha) \simeq \Phi(\beta)$), the sequences $\gamma_{\alpha}$ and $\gamma_{\beta}$ may not
coincide. For the real numbers in the Cantor's discontinuum, however, we have
the following.
\begin{proposition}[{Loeb \cite[Theorem 4.3]{Loeb05}}]
  \label{prop:CantorDiscontinuum}
  For any $\alpha \in \TTree$ and $\beta \in \Cantor$, 
  \[
    \Phi(\alpha) \simeq \kappa(\beta) 
    \imp \forall n \in \Nat \left(
    \gamma_{\alpha}(n) = \beta(n)\right).
  \]
\end{proposition}
\begin{proof}
  Fix $\alpha \in \TTree$ and $\beta \in \Cantor$, and suppose that
    $\Phi(\alpha) \simeq \kappa(\beta)$. We show 
    \[
      \forall n \in \Nat \left(
      \gamma_{\alpha}(n) = \beta(n)\right)
    \]
    by course of value induction on $n$. For the base case $(n=0)$,
    suppose that
    $\gamma_{\alpha}(0) \neq \beta(0)$. Then, either
    $\gamma_{\alpha}(0) = 1 \wedge \beta(0) = 0$ or 
    $\gamma_{\alpha}(0) = 0 \wedge \beta(0)=1$. 
    \smallskip
    
    \noindent\emph{Case $\gamma_{\alpha}(0) = 1 \wedge \beta(0) = 0$}:
    Then, $\fst{\CSet_{\seq{1}}} \leq \snd{\RInt_{\overline{\alpha}L(1)}}$, and so 
    $\snd{\CSet_{\overline{\beta}1}} <
    \fst{\RInt_{\overline{\alpha}L(1)}}$ by the definition of $L$. This contradicts $\Phi(\alpha) \simeq \kappa(\beta)$.
    \smallskip

    \noindent\emph{Case $\gamma_{\alpha}(0) = 0 \wedge \beta(0) = 1$}:
    Then, 
    $ \snd{\RInt_{\overline{\alpha}L(1)}} 
    < \fst{\CSet_{\overline{\beta}1}},$
    which contradicts $\Phi(\alpha) \simeq \kappa(\beta)$.

    The proof of the inductive case is similar.
\end{proof}

\begin{definition}
  For each $n \in \Nat$, define a binary relation $<_{n}$
  on $\Bin^{n}$ inductively as follows:
  \[
    \frac{}{\neg (\nil <_{0} \nil )}, 
    \qquad
    \frac{s \in \Bin^{n}}{s * \seq{0} <_{n+1} s * \seq{1}},
    \qquad
    \frac{s <_{n} t}{s * \seq{1} <_{n+1} t * \seq{0}}.
  \]
  When $s <_{n} t$, we say that $s$ is an \emph{immediate predecessor} of
  $t$ and $t$ is an \emph{immediate successor} of $s$. 
\end{definition}
The following lemmas and corollaries are for Proposition
\ref{prop:UCTcImpDFAN}. 
\begin{lemma}
  \label{lem:CharacterisationImmediateSucc}
  For each $n \in \Nat$ and $s,t \in \Bin^{n}$, 
  \[
    s <_{n} t \imp \exists u \in \BSeq \exists m \in \Nat
    \left( s = u * \seq{0} * 1^{m} \wedge t = u * \seq{1} *
    0^{m}\right).
  \]
\end{lemma}
\begin{proof}
  By induction on $n$.
\end{proof}

\begin{lemma}
  \label{lem:ImmediatePredSucc}
  For each $\alpha \in \TTree$, $n \in \Nat$, and $s \in
  \Bin^{n}$, 
  \begin{enumerate}
    \item\label{lem:ImmediatePredSucc1} 
      $s <_{n} \overline{\gamma_{\alpha}}n \imp \snd{\CSet_{s}} <
      \fst{\RInt_{\overline{\alpha}L(n)}}$,
    \item\label{lem:ImmediatePredSucc2}
      $ \overline{\gamma_{\alpha}}n <_{n} s \imp 
      \snd{\RInt_{\overline{\alpha}L(n)}} < \fst{\CSet_{s}}$.
  \end{enumerate}
\end{lemma}
\begin{proof}
  For \ref{lem:ImmediatePredSucc1},
  suppose that $s <_{n} \overline{\gamma_{\alpha}}n$. 
  By Lemma \ref{lem:CharacterisationImmediateSucc}, there
  exist $u \in \BSeq$ and  $m \in \Nat$ such that 
  $\overline{\gamma_{\alpha}}n = u * \seq{1} * 0^{m}$
  and $s = u * \seq{0} * 1^{m}$.
  Since $\snd{\CSet_{u*\seq{0}}} <
  \fst{\RInt_{\overline{\alpha}L(|u|+1)}}$, we have
  $\snd{\CSet_{s}} = \snd{\CSet_{u*\seq{0}}} <
  \fst{\RInt_{\overline{\alpha}L(|u|+1)}} \leq
  \fst{\RInt_{\overline{\alpha}L(n)}}$.
  The proof of \ref{lem:ImmediatePredSucc2} is similar.
\end{proof}
  Let $<_{n}^{+}$ denote the transitive closure of $<_{n}$.
  By induction on $n \in \Nat$, one can show that for
  any $s,t\in \Bin^{n}$, either
  $s = t$, $s <_{n}^{+} t$, or $t <_{n}^{+} s$.

\begin{corollary}
  \label{cor:ImmediatePredSucc}
  For each $\alpha \in \TTree$, $n \in \Nat$, and $s,t,u \in
  \Bin^{n}$, 
  \begin{enumerate}
    \item\label{cor:ImmediatePredSucc1}
      $s <_{n} t \wedge \snd{\CSet_{s}} < \fst{\RInt_{\overline{\alpha}L(n)}} < 
      \snd{\RInt_{\overline{\alpha}L(n)} }
      < \fst{\CSet_{t}} \to s = \overline{\gamma_{\alpha}}n \vee
      \overline{\gamma_{\alpha}}n = t$.

    \item\label{cor:ImmediatePredSucc2}
      $\text{$\CSet_{u} \approx \RInt_{\overline{\alpha}L(n)}$} \to u =
      \overline{\gamma_{\alpha}}n$.
  \end{enumerate}
\end{corollary}
\begin{proof}
  \noindent \ref{cor:ImmediatePredSucc1}.
  Suppose that $s <_{n} t$ and 
  $\snd{\CSet_{s}} < \fst{\RInt_{\overline{\alpha}L(n)}} < 
      \snd{\RInt_{\overline{\alpha}L(n)} }
      < \fst{\CSet_{t}}$.
  By Lemma~\ref{lem:ImmediatePredSucc}, 
  we must have $\overline{\gamma_{\alpha}}n = s$ or 
  $s <_{n}^{+} \overline{\gamma_{\alpha}}n$, and 
  $\overline{\gamma_{\alpha}}n = t$ or 
  $\overline{\gamma_{\alpha}}n <_{n}^{+} t$.
  But $s <_{n}^{+} \overline{\gamma_{\alpha}}n$
  and $\overline{\gamma_{\alpha}}n <_{n}^{+} t$
  is impossible because $s <_{n}^{+} \overline{\gamma_{\alpha}}n$
  if and only if $\overline{\gamma_{\alpha}}n = t$ or 
  $t <_{n}^{+}\overline{\gamma_{\alpha}}n$.
  \smallskip

  \noindent \ref{cor:ImmediatePredSucc2}.
  Suppose that $\CSet_{u} \approx \RInt_{\overline{\alpha}L(n)}$.
  If $\overline{\gamma_{\alpha}}n <_{n}^{+} u$, then 
  $\snd{\RInt_{\overline{\alpha}L(n)}} < \fst{\CSet_{u}}$ by 
  Lemma~\ref{lem:ImmediatePredSucc}, a contradiction. Similarly
  $u <_{n}^{+} \overline{\gamma_{\alpha}}n$ leads to a contradiction.
  Hence $u =  \overline{\gamma_{\alpha}}n$.
\end{proof}

\begin{lemma}
  \label{lem:BetweenImmediateSuccPred}
  Let  $n \in \Nat$, and let $s,t,u \in \Bin^{n}$ be
  such that $s <_{n} u <_{n} t$. Let $\alpha \in \TTree$.
  \begin{enumerate}

    \item \label{lem:BetweenImmediateSuccPred1}
      If $\snd{\CSet_{s}} < \Phi(\alpha) \leq \snd{\CSet_{u}}$,
      then one of the following holds:
      \begin{enumerate}
        \item 
        $
        \overline{\gamma_{\alpha}}n = s$ and $\gamma_{\alpha} =
       \widebreve{\overline{\gamma_{\alpha}}n}$;
        \item 
        $\overline{\gamma_{\alpha}}n = u$.
      \end{enumerate}

    \item \label{lem:BetweenImmediateSuccPred2}
      If $\fst{\CSet_{u}} \leq \Phi(\alpha) < \fst{\CSet_{t}} $,
      then one of the following holds:
      \begin{enumerate}
        \item 
        $\overline{\gamma_{\alpha}}n = u$;
      \item
        $
        \overline{\gamma_{\alpha}}n = t$ and $\gamma_{\alpha} =
       \widehat{\overline{\gamma_{\alpha}}n}$.
      \end{enumerate}

    \item \label{lem:BetweenImmediateSuccPred3}
      If $\snd{\CSet_{s}} <  \Phi(\alpha) < \fst{\CSet_{t}} $,
      then one of the following holds:
      \begin{enumerate}
        \item 
        $
        \overline{\gamma_{\alpha}}n = s$ and $\gamma_{\alpha} =
       \widebreve{\overline{\gamma_{\alpha}}n}$;
        \item 
        $\overline{\gamma_{\alpha}}n = u$;

        \item 
        $
        \overline{\gamma_{\alpha}}n = t$ and $\gamma_{\alpha} =
       \widehat{\overline{\gamma_{\alpha}}n}$.
      \end{enumerate}
  \end{enumerate}
\end{lemma}
\begin{proof}
 \ref{lem:BetweenImmediateSuccPred1}. Suppose
 that
  $\snd{\CSet_{s}} < \Phi(\alpha) \leq \snd{\CSet_{u}}$.
  Then, 
    $
    \snd{\CSet_{s}} < \fst{\RInt_{\overline{\alpha}L(m)}} \leq
    \snd{\CSet_{u}}
    $
  for sufficiently large $m \geq n$.
  Putting $N = m - n$, we have
  \[
    \snd{\CSet_{s*1^{N}}}  = \snd{\CSet_{s}}
    <  \fst{\RInt_{\overline{\alpha}L(m)}} 
    \leq \snd{\CSet_{u}}.
  \]
  Since $\snd{\RInt_{\overline{\alpha}L(m)}}
           < \fst{\CSet_{u}}$ or $\fst{\CSet_{u}} \leq \snd{\RInt_{\overline{\alpha}L(m)}}$, one of the following holds:
  \begin{enumerate}
    \item\label{lem:BetweenImmediateSuccPredProof1} 
      $\snd{\CSet_{s * 1^{N}}} 
           < \fst{\RInt_{\overline{\alpha}L(m)}}
           < \snd{\RInt_{\overline{\alpha}L(m)}}
           < \fst{\CSet_{u}} = \fst{\CSet_{u*0^{N}}}$;

   \item\label{lem:BetweenImmediateSuccPredProof2}
     $\RInt_{\overline{\alpha}L(m)} \approx \CSet_{u}$.
  \end{enumerate}
  In the case \ref{lem:BetweenImmediateSuccPred1}, either
  $s*1^{N} = \overline{\gamma_{\alpha}}m$ or
  $\overline{\gamma_{\alpha}}m = u * 0^{N}$
  by Corollary \ref{cor:ImmediatePredSucc}.
  If $s*1^{N} = \overline{\gamma_{\alpha}}m$,
  then by the construction of $\gamma_{\alpha}$ and 
  the fact 
  $\snd{\CSet_{\overline{\gamma_{\alpha}}m}}
   = \snd{\CSet_{s * 1^{N}}}
   < \fst{\RInt_{\overline{\alpha}L(m)}}$,
   we must have
  $\gamma_{\alpha} = \widebreve{\overline{\gamma_{\alpha}}m} =
  \widebreve{s} = \widebreve{\overline{\gamma_{\alpha}}n}$.
  If $\overline{\gamma_{\alpha}}m = u * 0^{N}$, then
  $\overline{\gamma_{\alpha}}n = u$.
  In the case \ref{lem:BetweenImmediateSuccPred2}, 
  we have $\RInt_{\overline{\alpha}L(n)} \approx \CSet_u$ as
  well. Thus $\overline{\gamma_{\alpha}}n = u$
  by Corollary~\ref{cor:ImmediatePredSucc}.
  \smallskip

\noindent\ref{lem:BetweenImmediateSuccPred2}.
  The proof is similar to \ref{lem:BetweenImmediateSuccPred1}.
  \smallskip

\noindent \ref{lem:BetweenImmediateSuccPred3}.
  If $\snd{\CSet_{s}} <  \Phi(\alpha) < \fst{\CSet_{t}}$,
  then either 
  $\snd{\CSet_{s}} < \Phi(\alpha) \leq \snd{\CSet_{u}}$
  or
  $\fst{\CSet_{u}} \leq \Phi(\alpha) < \fst{\CSet_{t}} $
  Then, the desired conclusion follows from 
  \ref{lem:BetweenImmediateSuccPred1} and 
  \ref{lem:BetweenImmediateSuccPred2}.
\end{proof}

\begin{corollary}
  \label{cor:EqualPathInitialSegment}
 Let $\alpha, \beta \in \TTree$ be such that $\Phi(\alpha) \simeq
 \Phi(\beta)$. For each $n \in \Nat$, one of the following holds:
 \begin{enumerate}
   \item  $\overline{\gamma_{\alpha}}n <_{n} \overline{\gamma_{\beta}}n$,
     $\gamma_{\alpha} = \widebreve{\overline{\gamma_{\alpha}}n}$, and 
     $\gamma_{\beta} = \widehat{\overline{\gamma_{\beta}}n}$;

   \item $\overline{\gamma_{\alpha}}n = \overline{\gamma_{\beta}}n$;

   \item $\overline{\gamma_{\beta}}n <_{n}
     \overline{\gamma_{\alpha}}n$,
     $\gamma_{\beta} = \widebreve{\overline{\gamma_{\beta}}n}$, and 
     $\gamma_{\alpha} = \widehat{\overline{\gamma_{\alpha}}n}$.
 \end{enumerate}
\end{corollary}
\begin{proof}
  Fix $n \in \Nat$. First, assume 
  that $\overline{\gamma_{\alpha}}n$ has an immediate predecessor
  $s$ and an immediate successor $t$. Then
  \[
    \snd{\CSet_{s}} < \fst{\RInt_{\overline{\alpha}L(n)}}
    \leq  \Phi(\alpha) \simeq \Phi(\beta)
    \leq \snd{\RInt_{\overline{\alpha}L(n)}}
    < \fst{\CSet_{t}}
  \]
  by Lemma \ref{lem:ImmediatePredSucc}.
  By the item \ref{lem:BetweenImmediateSuccPred3} of Lemma
  \ref{lem:BetweenImmediateSuccPred} one of the following holds:
  \begin{enumerate}
    \item
    $
    \overline{\gamma_{\beta}}n <_{n} \overline{\gamma_{\alpha}}n$
    and $\gamma_{\beta} =
   \widebreve{\overline{\gamma_{\beta}}n}$;
    \item

    $\overline{\gamma_{\beta}}n = \overline{\gamma_{\alpha}}n$;

    \item
    $
    \overline{\gamma_{\alpha}}n <_{n} \overline{\gamma_{\beta}}n$
    and $\gamma_{\beta} = \widehat{\overline{\gamma_{\beta}}n}$.
  \end{enumerate}
  If $\overline{\gamma_{\alpha}}n$ does not have an immediate predecessor
  or an immediate successor (or both), we obtain the same conclusion 
  using item \ref{lem:BetweenImmediateSuccPred1}
  or item \ref{lem:BetweenImmediateSuccPred2}
  of Lemma \ref{lem:BetweenImmediateSuccPred}
  (or trivially in case $\overline{\gamma_{\alpha}}n$ does not have
  both).

  Exchanging the role of $\alpha$ and $\beta$, we also have one 
  of the following:
  \begin{enumerate}
    \item
    $
    \overline{\gamma_{\alpha}}n <_{n} \overline{\gamma_{\beta}}n$
    and $\gamma_{\alpha} =
   \widebreve{\overline{\gamma_{\alpha}}n}$;
    \item

    $\overline{\gamma_{\alpha}}n = \overline{\gamma_{\beta}}n$;

    \item
    $
    \overline{\gamma_{\beta}}n <_{n} \overline{\gamma_{\alpha}}n$
    and $\gamma_{\alpha} = \widehat{\overline{\gamma_{\alpha}}n}$.
  \end{enumerate}
  Since conditions
$\overline{\gamma_{\alpha}}n < \overline{\gamma_{\beta}}n$,
$\overline{\gamma_{\alpha}}n = \overline{\gamma_{\beta}}n$, and 
$\overline{\gamma_{\beta}}n < \overline{\gamma_{\alpha}}n$
are mutually exclusive, we obtain the desired conclusion.
\end{proof}

\begin{proposition}
  \label{prop:UCTcImpDFAN}
  $\UCTcT$ implies $\DFAN$.
\end{proposition}
\begin{proof}
  Assume $\UCTcT$, and let $B \subseteq \BSeq$ be a decidable bar.
  Without loss of generality, assume $\nil \notin B$; otherwise,  $B$
  is trivially uniform. Before proceeding further, we introduce some
  notations: for $\alpha \in \TTree$ and $s \in \BSeq$, define
  \[
    \begin{aligned}
      N_{\alpha} 
      &\defeql
      \text{the least $n \in \Nat$ such that
        $\overline{\gamma_{\alpha}}n \in B$},\\
      \xi_s
      &\defeql
      \text{the least $n \in \Nat$ such that $\overline{\widebreve{s}}
    n \in B$}, \\
      \psi_s
      &\defeql
      \text{the least $n \in \Nat$ such that $\overline{\widehat{s}} n
    \in B$}. \\
    \end{aligned}
  \]
  For rationals $p,q \in \Rat$ such that  $p \leq q$, define
  \[
    [p,q]_{\Rat}
    \defeql
    \left\{ r \in \Rat \mid p \leq r \leq q \right\}.
  \]

  We construct a function $f \colon \UInt \to \Real$ as follows.  
  First, we define $f_{T} \colon \TTree \allowbreak \to \Real$.
  Fix $\alpha \in \TTree$. Without loss of generality,
  assume that $\overline{\gamma_{\alpha}}N_{\alpha}$ has an immediate
  predecessor $s$ and an immediate successor $t$.
  Define a piecewise linear function $F_{\alpha} \colon [\snd{\CSet_{s}},
  \fst{\CSet_{t}}]_{\Rat} \to \Rat$ as follows (see Figure~\ref{fig:F}):%
  \footnote{
  If $\overline{\gamma_{\alpha}}N_{\alpha}$ does not have an immediate
  predecessor or an immediate successor, then we simply restrict the domain
  of $F_{\alpha}$ to
  $[\fst{\CSet_{\overline{\gamma_{\alpha}}N_{\alpha}}}, \fst{\CSet_{t}}]_{\Rat}$
  or $[\snd{\CSet_{s}},
  \snd{\CSet_{\overline{\gamma_{\alpha}}N_{\alpha}}}]_{\Rat}$ respectively.}
  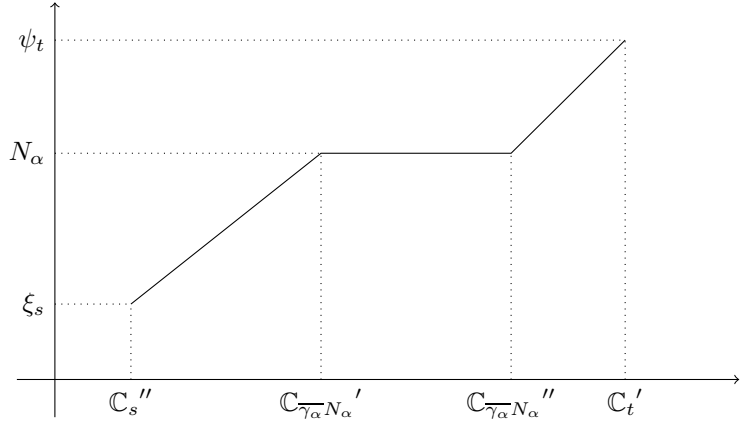
\begin{figure}[bt]
    \centering
  \begin{tikzpicture}
    \draw [->] (-.5,0) -- (9,0);
    \draw [->] (0,-.5) -- (0,5);
    % Graph of F
    \draw (1,1) -- (3.5,3) -- (6,3) -- (7.5,4.5);
    % Labels on x-axis
    \draw [anchor=north] (1,0) node {$\snd{\CSet_{s}}$};
    \draw [dotted] (1,0) -- (1,1);
    \draw [anchor=north] (3.5,0) node
    {$\fst{\CSet_{\overline{\gamma_{\alpha}}N_{\alpha}}}$};
    \draw [dotted] (3.5,0) -- (3.5,3);
    \draw [anchor=north] (6,0) node 
    {$\snd{\CSet_{\overline{\gamma_{\alpha}}N_{\alpha}}}$};
    \draw [dotted] (6,0) -- (6,3);
    \draw [anchor=north] (7.5,0) node {$\fst{\CSet_{t}}$};
    \draw [dotted] (7.5,0) -- (7.5,4.5);
    % Labels on y-axis
    \draw [anchor=east] (0,1) node {$\xi_s$};
    \draw [dotted] (0,1) -- (1,1);
    \draw [anchor=east] (0,3) node {$N_{\alpha}$};
    \draw [dotted] (0,3) -- (3.5,3);
    \draw [anchor=east] (0,4.5) node {$\psi_t$};
    \draw [dotted] (0,4.5) -- (7.5,4.5);
    % Some more decoration
  \end{tikzpicture}
  \caption{The graph of $F_{\alpha}$} \label{fig:F}
  \end{figure}
  \begin{equation}
    \label{def:F}
    F_{\alpha}(r) \defeql 
    \begin{cases}
      \frac{N_{\alpha} - \xi_s}%
      {\fst{\CSet_{\overline{\gamma_{\alpha}}N_{\alpha}}} - \snd{\CSet_s}}
      (r - \snd{\CSet_s}) + \xi_s
      & \text{if $r \in [\snd{\CSet_{s}},
      \fst{\CSet_{\overline{\gamma_{\alpha}}N_{\alpha}}}]$}, \\[.5em]
      N_{\alpha}
      & \text{if $r \in \CSet_{\overline{\gamma_{\alpha}}N_{\alpha}}$}, \\[.5em]
      \frac{\psi_t  - N_{\alpha}}%
      {\fst{\CSet_{t}} - \snd{\CSet_{\overline{\gamma_{\alpha}}N_{\alpha}}}}
      (r - \snd{\CSet_{\overline{\gamma_{\alpha}}N_{\alpha}}}) +
      N_{\alpha}
      & \text{if $r \in
        [\snd{\CSet_{\overline{\gamma_{\alpha}}N_{\alpha}}},
        \fst{\CSet_{t}}]$}.
      \end{cases}
    \end{equation}
  By Lemma \ref{lem:ImmediatePredSucc}, we have 
  \[
    \snd{\CSet_{s}} 
    < \fst{\RInt_{\overline{\alpha}L(N_{\alpha})}}
    < \snd{\RInt_{\overline{\alpha}L(N_{\alpha})}}
    < 
    \fst{\CSet_{t}},
  \]
   so the sequence $\seq{x^{n}_{\alpha}}_{n \geq L(N_{\alpha})}$ 
   lies in the domain of $F_{\alpha}$. Define
  \[
    f_{T}(\alpha) 
    \defeql \seq{F_{\alpha}(x^{n}_{\alpha})}_{n \geq L(N_{\alpha})}.
  \]
  Since $F_{\alpha}$ is piecewise linear (and thus uniformly continuous), 
  the sequence $\seq{F_{\alpha}(x^{n}_{\alpha})}_{n \geq L(N_{\alpha})}$ is a
  fundamental sequence. 

  Next, we show that 
  \begin{equation}
  \label{eq:welldeffT}
    \Phi(\alpha) \simeq \Phi(\beta) \to f_{T}(\alpha) \simeq
    f_{T}(\beta)
   \end{equation}
  for all $\beta \in \TTree$. Let $\beta \in \TTree$ such that
  $\Phi(\beta) \simeq \Phi(\alpha)$. Assume, without loss of
  generality, that $\overline{\gamma_{\beta}}N_{\beta}$
  has an immediate predecessor $s'$ and an immediate successor
  $t'$. Define a piecewise linear function 
  $F_{\beta} \colon  [\snd{\CSet_{s'}}, \fst{\CSet_{t'}}]_{\Rat} \to \Rat$
  by \eqref{def:F} using  $\beta,s',t'$
  instead of $\alpha, s, t$.
  By Corollary \ref{cor:EqualPathInitialSegment}, one of the following
  holds:
   \begin{enumerate}
     \item\label{prop:UCTcImpDFAN1}
       $\overline{\gamma_{\alpha}}N_{\alpha}<_{N_{\alpha}}
       \overline{\gamma_{\beta}}N_{\alpha}$,
       $\gamma_{\alpha} = \widebreve{\overline{\gamma_{\alpha}}N_{\alpha}}$, and 
       $\gamma_{\beta} = \widehat{\overline{\gamma_{\beta}}N_{\alpha}}$;

     \item\label{prop:UCTcImpDFAN2}
       $\overline{\gamma_{\alpha}}N_{\alpha} =
       \overline{\gamma_{\beta}}N_{\alpha}$;

     \item\label{prop:UCTcImpDFAN3}
       $\overline{\gamma_{\beta}}N_{\alpha} <_{N_{\alpha}}
       \overline{\gamma_{\alpha}}N_{\alpha}$,
       $\gamma_{\beta} = \widebreve{\overline{\gamma_{\beta}}N_{\alpha}}$, and 
       $\gamma_{\alpha} = \widehat{\overline{\gamma_{\alpha}}N_{\alpha}}$.
   \end{enumerate}
   By the same corollary, one of the following holds:
   \begin{enumerate}[label={\arabic*}$'$., ref={\arabic*}$'$]
     \item\label{prop:UCTcImpDFAN1p}
       $\overline{\gamma_{\alpha}}N_{\beta}<_{N_{\beta}}
       \overline{\gamma_{\beta}}N_{\beta}$,
       $\gamma_{\alpha} = \widebreve{\overline{\gamma_{\alpha}}N_{\beta}}$, and 
       $\gamma_{\beta} = \widehat{\overline{\gamma_{\beta}}N_{\beta}}$;

     \item\label{prop:UCTcImpDFAN2p}
       $\overline{\gamma_{\alpha}}N_{\beta} =
       \overline{\gamma_{\beta}}N_{\beta}$;

     \item\label{prop:UCTcImpDFAN3p}
       $\overline{\gamma_{\beta}}N_{\beta} <_{N_{\beta}}
       \overline{\gamma_{\alpha}}N_{\beta}$,
       $\gamma_{\beta} = \widebreve{\overline{\gamma_{\beta}}N_{\beta}}$, and 
       $\gamma_{\alpha} = \widehat{\overline{\gamma_{\alpha}}N_{\beta}}$.
   \end{enumerate}
   The only possible combinations are \ref{prop:UCTcImpDFAN1}
   and \ref{prop:UCTcImpDFAN1p}; \ref{prop:UCTcImpDFAN2} and
   \ref{prop:UCTcImpDFAN2p};  and 
  \ref{prop:UCTcImpDFAN3} and \ref{prop:UCTcImpDFAN3p}.
   In the case \ref{prop:UCTcImpDFAN1} and \ref{prop:UCTcImpDFAN1p}, 
   we have
   $\overline{\gamma_{\beta}}N_{\alpha} = t$
   and so 
   $
   \overline{\widehat{t}}\psi_t
   = \overline{\gamma_{\beta}}\psi_t
   = \overline{\gamma_{\beta}}N_{\beta}
   $
   by the definitions of $\psi_t$ and $N_{\beta}$.
   Similarly, we have $s' = \overline{\gamma_{\alpha}}N_{\beta}$, and so
   $\overline{\widebreve{s'}}\xi_{s'}
   = \overline{\gamma_{\alpha}}\xi_{s'}
   = \overline{\gamma_{\alpha}}N_{\alpha}$.
   Since $s' <_{N_{\beta}} \overline{\gamma_{\beta}}N_{\beta}$,
   there exists $u \in
  \BSeq$ and $m \in \Nat$ such that $s' = u * \seq{0} * 1^{m}$
  and $\overline{\gamma_{\beta}}N_{\beta} = u * \seq{1} *
  0^{m}$. By the definition of $N_{\beta}$, we must have
  $\lh{u} < \xi_{s'}$. Then, 
  $\snd{\CSet_{s'}}
  = \snd{\CSet_{u*\seq{0}}}
  = \snd{\CSet_{\overline{\widebreve{u*\seq{0}}}\xi_{s'}}}
  = \snd{\CSet_{\overline{\widebreve{s'}}\xi_{s'}}}
  = \snd{\CSet_{\overline{\gamma_{\alpha}}N_{\alpha}}}
  $. Similarly, we have $\fst{\CSet_{t}} =
  \fst{\CSet_{\overline{\gamma_{\beta}}N_{\beta}}}$.
  Thus, the functions $F_{\alpha}$ and $F_{\beta}$
  agree on the interval
  $[
    \snd{\CSet_{\overline{\gamma_{\alpha}}N_{\alpha}}},
    \fst{\CSet_{\overline{\gamma_{\beta}}N_{\beta}}}
  ]$.
  Since 
  $\snd{\RInt_{\overline{\alpha}L(N_{\alpha})}}
  < \fst{\CSet_{t}}
  = \fst{\CSet_{\overline{\gamma_{\beta}}N_{\beta}}}$,
  $
   \snd{\CSet_{\overline{\gamma_{\alpha}}N_{\alpha}}}
  = \snd{\CSet_{s'}}
  < \fst{\RInt_{\overline{\beta}L(N_{\beta})}}$
  and $\Phi(\alpha) \simeq \Phi(\beta)$,
  for sufficiently large $N_{\alpha,\beta} \geq \max \left\{ L(N_{\alpha}),
  L(N_{\beta}) \right\}$ both
  $\RInt_{\overline{\alpha}N_{\alpha. \beta}}$ 
  and $\RInt_{\overline{\beta}N_{\alpha. \beta}}$
  lie in the interval 
  $[\snd{\CSet_{\overline{\gamma_{\alpha}}N_{\alpha}}},
  \fst{\CSet_{\overline{\gamma_{\beta}}N_{\beta}}}]$.
   Since $F_{\alpha}$ and $F_{\beta}$ are
  uniformly continuous (and hence preserve equality on regular
  sequences), we have
  \[
    \begin{aligned}
      \seq{F_{\alpha}(x^{n}_{\alpha})}_{n \geq L(N_{\alpha})}
      &\simeq
      \seq{F_{\alpha}(x^{n}_{\alpha})}_{n \geq N_{\alpha,\beta}} \\
      &\simeq
      \seq{F_{\alpha}(x^{n}_{\beta})}_{n \geq N_{\alpha,\beta}} \\
      &=
      \seq{F_{\beta}(x^{n}_{\beta})}_{n \geq N_{\alpha,\beta}}
      \simeq
      \seq{F_{\beta}(x^{n}_{\beta})}_{n \geq L(N_{\beta})}.
    \end{aligned}
  \]
   Thus $f_{T}(\alpha) \simeq f_{T}(\beta)$.
   In the case \ref{prop:UCTcImpDFAN2} and \ref{prop:UCTcImpDFAN2p},
   we must have $N_{\alpha} = N_{\beta}$. Then 
   $F_{\alpha}$ and $F_{\beta}$ agree, and so $f_{T}(\alpha) \simeq
   f_{T}(\beta)$.
   The case \ref{prop:UCTcImpDFAN3} and \ref{prop:UCTcImpDFAN3p} is symmetric to
   the first case.
   Therefore $f_{T}(\alpha) \simeq f_{T}(\beta)$.

   For an arbitrary regular sequence $x$ in $\UInt$, define
   \[
     f(x) \defeql f_{T}(\alpha_{x}),
   \]
   where $\alpha_{x} \in \TTree$ is the path determined by
   \eqref{def:FromRegularToPath}.
   Then, for any $x,y
   \in \UInt$ such that $x \simeq y$, we have $\Phi(\alpha_{x}) \simeq x \simeq y
   \simeq \Phi(\alpha_{y})$ by Proposition \ref{prop:Surject}. Then
   by \eqref{eq:welldeffT}, we have
   \[
     f(x) = f_{T}(\alpha_{x})\simeq f_{T}(\alpha_{y}) = f(y).
   \]
    Thus $f$ is a function
   from $\UInt$ and $\Real$.

  Next, we define a ternary modulus $g \colon \Nat \to \TTree \to \Nat$ of
  $f$ as follows. Fix $k \in \Nat$ and $\alpha \in \TTree$. Without loss of
  generality, assume that $\overline{\gamma_{\alpha}}N_{\alpha}$
  has an immediate predecessor $s$ and an immediate successor
  $t$.
  Let $N$ be the least $n \in \Nat$ such that 
  $\snd{\CSet_{s}}
  < \fst{\RInt_{\overline{\alpha}L(N_{\alpha})}} - 2^{-n}$
  and 
  $\snd{\RInt_{\overline{\alpha}L(N_{\alpha})}} + 2^{-n}
  < \fst{\CSet_{t}}$.
  Put
  \[
    g_{k}(\alpha) \defeql \max \left\{ N, \omega(k) + 1 \right\},
  \]
  where $\omega$ is a modulus of uniform continuity of the function
  $F_{\alpha}$ defined by \eqref{def:F}.
  Note that $g_{k} \colon \TTree \to \Nat$
  is continuous because the construction of $g_{k}(\alpha)$ depends only
  on the initial segment of $\alpha$ up to length $L(N_{\alpha})$.

  We show that $g$ is a ternary modulus of $f$. Fix $k \in \Nat$ and
  $\alpha \in \TTree$. Without loss of generality, assume
  that $\overline{\gamma_{\alpha}}N_{\alpha}$ has an immediate
  predecessor $s$ and an immediate successor $t$. 
  Let $x \in \UInt$ be such that
  $|\Phi(\alpha) - x| \leq 2^{-g_{k}(\alpha)}$.
  We may assume that $x = \Phi(\beta)$ for some $\beta \in \TTree$.
  Since 
  $\fst{\RInt_{\overline{\alpha}L(N_{\alpha})}}
  \leq 
  \snd{\Phi(\alpha) \leq \RInt_{\overline{\alpha}L(N_{\alpha})}}$,
  we have $\snd{\CSet_{s}} < \Phi(\beta) < \fst{\CSet_{t}}$.
  By Lemma \ref{lem:BetweenImmediateSuccPred}, one of the following
  holds:
  \begin{enumerate}
    \item \label{prop:UCTcImpDFANg1}
    $
    \overline{\gamma_{\beta}}N_{\alpha} = s$ and $\gamma_{\beta} =
   \widebreve{\overline{\gamma_{\beta}}N_{\alpha}}$;
    \item \label{prop:UCTcImpDFANg2}
    $\overline{\gamma_{\beta}}N_{\alpha} 
    = \overline{\gamma_{\alpha}}N_{\alpha}$;

    \item \label{prop:UCTcImpDFANg3}
    $
    \overline{\gamma_{\beta}}N_{\alpha} = t$ and $\gamma_{\beta} =
   \widehat{\overline{\gamma_{\beta}}N_{\alpha}}$.
  \end{enumerate}
  In the case \ref{prop:UCTcImpDFANg1}, we have
  $\overline{\gamma_{\beta}}N_{\beta} = \overline{\breve{s}}\, \xi_s$. On the other hand, since $s <_{N_{\alpha}}
  \overline{\gamma_{\alpha}}N_{\alpha}$, there exists $u \in
  \BSeq$ and $m \in \Nat$ such that $s = u * \seq{0} * 1^{m}$
  and $\overline{\gamma_{\alpha}}N_{\alpha} = u * \seq{1} *
  0^{m}$. By the definition of $N_{\alpha}$, we must have
  $\lh{u} < N_{\beta}$. Then, $t' = \overline{\widehat{u * \seq{1}}}N_{\beta}$ is  an immediate successor of
  $\overline{\gamma_{\beta}}N_{\beta}$ and $\overline{\widehat{t'}} \psi_{t'} = \overline{\gamma_{\alpha}}N_{\alpha}$.
  Thus, the functions $F_{\alpha}$ and $F_{\beta}$ determined
  by $\alpha$ and $\beta$ as in \eqref{def:F} agree on the 
  interval 
  $[
    \snd{\CSet_{\overline{\gamma_{\beta}}N_{\beta}}},
    \fst{\CSet_{\overline{\gamma_{\alpha}}N_{\alpha}}}
  ]$.
  Since $\snd{\CSet_{\overline{\gamma_{\beta}}N_{\beta}}}
  = \snd{\CSet_{s}} < \Phi(\beta)$ and
  $\snd{\RInt_{\overline{\beta}L(N_{\beta})}}
  < \fst{\CSet_{t'}}
  = \fst{\CSet_{\overline{\gamma_{\alpha}}N_{\alpha}}}$, 
  the term $x_{\beta}^{n}$ lies in 
  $[
    \snd{\CSet_{\overline{\gamma_{\beta}}N_{\beta}}},
    \fst{\CSet_{\overline{\gamma_{\alpha}}N_{\alpha}}}
  ]$
  for sufficiently large $n \geq M_1$ for some
  $M_{1} \in \Nat$. Since 
    $
    | \Phi(\alpha) - \Phi(\beta) | \leq 2^{-(\omega(k) +1)},
    $
   we have
  \[
    \forall m \in \Nat
    \left(  |x_{\alpha}^{M_{2} + m} - x_{\beta}^{M_{2} + m}| 
    \leq 2^{-\omega(k)}\right),
  \]
  where $M_{2} \defeql \max \{ M_{1}, \omega(k) + 1\}$.
  Thus 
    $
      |F_{\alpha}(x^{m}_{\alpha}) - F_{\beta}(x^{m}_{\beta})| 
    \leq 2^{-k} 
    $
  for all $m \geq M_{2}$, which implies
  \[
    |\seq{F_{\alpha}(x_{\alpha}^{n})}_{n \geq L(N_{\alpha})}
    - \seq{F_{\beta}(x_{\beta}^{n})}_{n \geq L(N_{\beta})} |
    \leq 2^{-k},
  \]
  that is, $|f(\Phi(\alpha)) - f(\Phi(\beta)) | \leq
  2^{-k}$.
  The cases \ref{prop:UCTcImpDFANg2} and \ref{prop:UCTcImpDFANg3}
  are treated similarly.  Thus, $g$ is a modulus of $f$.
  
  By $\UCTcT$, $f$ is uniformly continuous. Then the composition
  $f \circ \kappa \colon \Cantor \to \Real$
  of $f$ with $\kappa \colon
  \Cantor \to \UInt$ given by \eqref{def:kappa} is uniformly
  continuous as well.
  We show that
  \[
    f(\kappa(\alpha))  = \text{the least $n \in \Nat$ such that
      $\overline{\alpha}n \in B$}
  \]
  for all $\alpha \in \Cantor$.
  Fix $\alpha \in \Cantor$
  and choose $\beta \in \TTree$ such that $\kappa(\alpha) \simeq
  \Phi(\beta)$ (cf.\ Proposition \ref{prop:Surject}). Then,  $\alpha = \gamma_{\beta}$ by Proposition
  \ref{prop:CantorDiscontinuum}, so
  it suffices to show $f(\Phi(\beta)) \simeq N_{\beta}$.
  Without loss of generality, assume that
  $\overline{\gamma_{\beta}}N_{\beta}$ has an immediate predecessor
  $s$ and an immediate successor $t$.
  Since $\alpha = \gamma_{\beta}$
  and $\snd{\CSet_{s}} <
  \fst{\RInt_{\overline{\beta}L(N_{\beta})}}< 
  \snd{\RInt_{\overline{\beta}L(N_{\beta})}} < \fst{\CSet_{t}}$
  by Lemma \ref{lem:ImmediatePredSucc},
  the sequences
  $\seq{\kappa(\alpha)(n)}_{n \geq
    L(N_{\beta})}$ 
  and $\seq{x^{n}_{\beta}}_{n \geq
  L(N_{\beta})}$ lie in the domain of the function $F_{\beta}$ defined as in \eqref{def:F}.
  Since $\seq{\kappa(\alpha)(n)}_{n \in \Nat}
  \simeq \seq{x^{n}_{\beta}}_{n \in \Nat}$
  and $\seq{\kappa(\alpha)(n)}_{n \geq L(N_{\beta})}$
  lies in $\CSet_{\overline{\gamma_{\beta}}N_{\beta}}$,
  we have 
  \[
    f(\Phi(\beta)) \simeq  \seq{F_{\beta}(x^{n}_{\beta})}_{n \geq
      L(N_{\beta})} \simeq \seq{F_{\beta}(\kappa(\alpha)(n))}_{n \geq
        L(N_{\beta})} = N_{\beta},
  \]
  as required.
  Since $f \circ \kappa$ is uniformly continuous, there exists
  $M_{3} \in \Nat$ such that 
  \[
    \forall \alpha, \beta \in \Cantor \left(
    \overline{\alpha}M_{3} = \overline{\beta} M_{3}
    \imp f(\kappa(\alpha)) =  f(\kappa(\beta))\right).
  \]
  Put $M \defeql \max \left\{ f(\kappa(\widehat{s})) \mid s \in
  \Bin^{M_{3}}\right\}$. For any $\alpha \in \Cantor$, 
  we have $f(\kappa(\alpha)) =
  f(\kappa(\widehat{\overline{\alpha}M_{3}})) \leq M$.
  Therefore $B$ is uniform.
\end{proof}
\begin{theorem}
  \label{thm:EquivUCTcDFAN}
  $\UCTc$ and $\DFAN$ are equivalent.
\end{theorem}
\begin{proof}
  By Proposition \ref{prop:DFANimpliesUContMod}, Proposition
  \ref{prop:UCTcImpDFAN}, and Theorem \ref{cor:EqiuvUCTcUTCcT}.
\end{proof}

We summarise the equivalents of the decidable fan theorem.
\begin{theorem}\label{thm:EquivDFT}
  The principles
   $\DFAN$, $\UCc$, $\UCTcC$, and $\UCTc$
  are pairwise equivalent.
\end{theorem}
The equivalence of $\DFAN$ and $\UCc$ is due to
Berger~\cite[Proposition~4]{BergerFANandUC}. We have
established the other equivalence by showing $\UCc\leftrightarrow
\UCTcC$ (Proposition \ref{prop:ContMod}) and $\DFAN \leftrightarrow
\UCTc$ (Theorem \ref{thm:EquivUCTcDFAN}).
As our proof shows, it is not hard to show that $\DFAN$
implies the rest of the items in Theorem \ref{thm:EquivDFT}.
Among $\UCc$, $\UCTcC$, and $\UCTc$, the principle $\UCTcC$ seems to
be most general. Indeed, $\UCc$ immediately follows from $\UCTcC$.
Moreover, it is straightforward to show that $\UCTcC$ implies
$\UCTc$ using Theorem \ref{prop:UniversalQuotient}. Thus, the gist of our proof consists in
showing $\UCc \to \UCTcC$ and $\UCTc \to \DFAN$.

\section{Codes of continuous functions}\label{sec:Code}
In this section, we clarify the relation between type one continuous
functions described in Loeb \cite{Loeb05} and real-valued functions on
the unit interval which have continuous moduli.

Throughout this section, we assume that real numbers are represented 
by regular sequences.
For $x \in \Real$ and $k \in \Nat$, let $x_{k}$ denote the $k$-th term
of $x$.  We write $F \colon \LoebReal \to \Real$  for the bijection
between the set of shrinking sequences and the set of regular
sequences and $G \colon \Real \to \LoebReal$ for the inverse of $F$
(see Proposition~\ref{prop:OrderBijectLoebReal}).  Recall from
\eqref{def:RatInt} that $\RatInt$ denotes the set of (pairs of
end-points of) rational
intervals. In the following, we identify $\RatInt$ with a subset of
$\Nat$ via a fixed coding of rational numbers and the pairing
function.

\begin{definition}[{Loeb~\cite[Definition 3.1]{Loeb05}}]
  \label{def:CodeCont}
  A function $\varphi \Ter^{*} \to \Nat$ is a \emph{code of a
  continuous function} if
  \begin{enumerate}[({C}1)]
    \item\label{def:CodeCont1} $\forall s \in \Ter^{*} \left( \varphi(s) \neq 0 \imp
      \varphi(s) \dotminus 1  \in \RatInt \right)$,

    \item\label{def:CodeCont2} $\forall k \in \Nat \forall \alpha \in \TTree \exists n \in
      \Nat \left( \varphi(\overline{\alpha}n) \neq 0 \wedge
      \lh{\varphi(\overline{\alpha}n) \dotminus 1} \leq 2^{-k} \right)$,

    \item\label{def:CodeCont3} $\forall s \in \Ter^{*} \forall i \in \Ter 
      \left(\varphi(s) \neq 0 \imp \right. \\
      \phantom{\forall s \in \Ter^{*} \forall i \in \Ter \:}
      \left.\varphi(s * \seq{i}) \neq 0 \wedge
      \varphi(s*\seq{i}) \dotminus 1 \sqsubseteq \varphi(s) \dotminus
      1\right)$,

    \item\label{def:CodeCont4} $\forall s, t \in \Ter^{*} \left
      ( \varphi(s) \neq 0 \wedge \varphi(t) \neq 0 \wedge \RInt_s
      \approx \RInt_t \to \varphi(s) \dotminus 1 \approx \varphi(t)
      \dotminus 1\right)$.
  \end{enumerate}
\end{definition}
\begin{remark}
  Loeb~\cite{Loeb05} calls a code of a continuous function by
  \emph{continuous function}. Here, we introduce an alternative
  terminology in order to avoid any possible confusion with the usual
  notion of (pointwise) continuity for real-valued functions.
  In what follows, we call a code of a continuous function
  simply by \emph{code}.
\end{remark}

Given a code $\varphi \colon \Ter^{*} \to \Nat$, define $f^{\varphi}_{T} \colon
\TTree \to \Nat \to \RatInt$ by 
\begin{gather}
  \notag
  f^{\varphi}_{T}(\alpha) \defeql \seq{\varphi(\overline{\alpha}h_{n}(\alpha))
\dotminus 1}_{n \in \Nat},\\
\shortintertext{where}
\label{eq:hn}
h_{k}(\alpha) 
\defeql
\text{the least $n \in \Nat$ such that $\varphi(\overline{\alpha}n)
\neq 0 \wedge\lh{\varphi(\overline{\alpha}n) \dotminus 1} \leq
2^{-k}$}.
\end{gather}
Note that $h_{k}(\alpha)$ exists by the property
\ref{def:CodeCont2} of $\varphi$.
\begin{lemma}
  \label{lem:Realvalued}
  For each $\alpha,\beta \in \TTree$, 
  \begin{enumerate}
    \item\label{lem:Realvalued1} $f^{\varphi}_{T}(\alpha) \in \LoebReal$, 
    \item\label{lem:Realvalued2} $\Phi(\alpha) \simeq \Phi(\beta) \imp f^{\varphi}_{T}(\alpha) \simeq
      f^{\varphi}_{T}(\beta)$.
  \end{enumerate}
\end{lemma}
\begin{proof}
  \ref{lem:Realvalued1}. We must check \ref{def:LoebRea1}
  and \ref{def:LoebRea2} (cf.\ Definition~\ref{def:LoebReal}). 
  For \ref{def:LoebRea1}, by the leastness of $h_{n}(\alpha)$, we have
  $h_{n}(\alpha) \leq h_{n+1}(\alpha)$.  Thus 
  $\varphi(\overline{\alpha}h_{n+1}(\alpha)) \dotminus 1 
  \sqsubseteq \varphi(\overline{\alpha}h_{n}(\alpha)) \dotminus 1$ by
  \ref{def:CodeCont3}. The property \ref{def:LoebRea2} follows from
  \ref{def:CodeCont2}.
  \smallskip

  \noindent\ref{lem:Realvalued2}. 
  Suppose that $\Phi(\alpha) \simeq \Phi(\beta)$. Then
  $\RInt_{\overline{\alpha}n} \approx \RInt_{\overline{\beta}m}$ for
  all $n,m \in \Nat$, so
  $\varphi(\overline{\alpha}h_{n}(\alpha)) \dotminus 1 
  \approx \varphi(\overline{\beta}h_{n}(\beta)) \dotminus 1$ 
  for all $n \in \Nat$ by \ref{def:CodeCont4}.
  Thus
  $f^{\varphi}_{T}(\alpha) \simeq f^{\varphi}_{T}(\beta)$.
\end{proof}
Define a function $f_{\varphi} \colon \UInt \to \Real$ by 
\begin{equation}
  \label{eq:PhiToFunc}
  f_{\varphi}(x) \defeql F(f^{\varphi}_{T}(\alpha_{x})),
\end{equation}
where $\alpha_{x}$ is the path determined by $x$ as in
\eqref{def:FromRegularToPath}.  Since $x \simeq \Phi(\alpha_{x})$ by
Proposition \ref{prop:Surject}, $f_{\varphi}$ preserves the equality
of $\Real$ by Lemma \ref{lem:Realvalued}. Hence $f_{\varphi}$ is a function from $\UInt$ to $\Real$.

\begin{definition}
  A function $f \colon \UInt \to \Real$ is said to be \emph{induced}
  by a code $\varphi \colon \Ter^{*} \to \Nat$ if $\forall x \in \UInt
  \left( f(x) \simeq f_{\varphi}(x) \right)$.
\end{definition}
Note that a function $f \colon \UInt \to \Real$ is induced by a code
$\varphi \colon \Ter^{*} \to \Nat$  if and only if
\[
  \forall \alpha \in \TTree \left( G(f(\Phi(\alpha))) \simeq
  f^{\varphi}_{T}(\alpha) \right).
\]

\begin{lemma}
  \label{lem:CodeToContMod}
  If $f \colon \UInt \to \Real$ is induced by a code, then $f$
  has a continuous modulus.
\end{lemma}
\begin{proof}
  By Proposition \ref{prop:EquivModTernMod},
  it suffices to show that the function $f_{\varphi}$ induced by a
  code $\varphi \colon \Ter^{*} \to \Nat$ as in \eqref{eq:PhiToFunc} 
  has a continuous ternary modulus. 
  Define $g \colon \Nat \to \TTree \to \Nat$ by
  \begin{equation}
    \label{eq:CodeTernaryMod}
    g_{k}(\alpha) \defeql h_{k}(\rho(\alpha)) + 6,
  \end{equation}
  where $h_{k}(\alpha)$ 
  and $\rho \colon \TTree \to \TTree$ are  defined by \eqref{eq:hn}
  and \eqref{def:rho} respectively.
  We show that $g$ is a continuous ternary modulus of $f_{\varphi}$.
  First, note that $\rho$ is uniformly continuous.
  It is also easy to see that $h_{k} \colon \TTree \to \Nat$ is a
  continuous modulus of itself. Thus $g_{k}$ is continuous for each $k \in
  \Nat$. To see that $g$ is a ternary modulus of
  $f_{\varphi}$, it suffices to show that 
  \[
     \lh{\Phi(\alpha) - \Phi(\beta)} \leq 2^{-g_{k}(\alpha)}
    \imp \lh{f^{\varphi}_{T}(\alpha) - f^{\varphi}_{T}(\beta)} \leq
    2^{-k} 
  \]
  for all $k \in \Nat$ and $\alpha,\beta \in \TTree$.
  Fix $k \in \Nat$ and $\alpha,\beta \in \TTree$, and suppose that 
  $\lh{\Phi(\alpha) - \Phi(\beta)} \leq 2^{-g_{k}(\alpha)}$.
  Then $\lh{\Phi(\alpha) - \Phi(\beta)} < 2^{-(h_{k}(\rho(\alpha))+
5)}$. By Proposition \ref{prop:Quotient}, there exists $\gamma \in
  \overline{\rho(\alpha)}h_{k}(\rho(\alpha))$ such that
  $\Phi(\gamma) \simeq \Phi(\beta)$. By the definition of
  $h_{k}(\rho(\alpha))$, we have
    $
    h_{k}(\gamma) = h_{k}(\rho(\alpha)).
    $
 Thus
  \[
    \begin{aligned}
      f^{\varphi}_{T}(\gamma)_{k}
      =  \varphi(\overline{\gamma}h_{k}(\gamma)) \dotminus 1 
      =  \varphi(\overline{\gamma}h_{k}(\rho(\alpha))) \dotminus 1 
      =  \varphi(\overline{\rho(\alpha)}h_{k}(\rho(\alpha))) \dotminus 1 
      = f^{\varphi}_{T}(\rho(\alpha))_{k}.
    \end{aligned}
  \]
  Since $\lh{f^{\varphi}_{T}(\rho(\alpha))_{k}} \leq 2^{-k}$, we have
  $
  \lh{f^{\varphi}_{T}(\rho(\alpha)) - f^{\varphi}_{T}(\gamma)} \leq
  2^{-k}.
  $
  Therefore
  $
  \lh{f^{\varphi}_{T}(\alpha) - f^{\varphi}_{T}(\beta)} \leq
  2^{-k}
  $
  by Corollary \ref{cor:rhoPreserveEqual} and Lemma \ref{lem:Realvalued}.
\end{proof}
  
To prove the converse of Lemma \ref{lem:CodeToContMod}, we
use the following lemma.
\begin{lemma}
  \label{lem:ContItselfMonotone}
  If $f \colon \UInt \to \Real$ has a continuous ternary modulus,
  then $f$ has a continuous ternary modulus $g \colon \Nat \to \TTree
  \to \Nat$ such that for each $k \in \Nat$, 
  \begin{enumerate}
    \item\label{lem:ContItselfMonotone1} $g_{k}$ is a continuous modulus of itself,

    \item\label{lem:ContItselfMonotone2} $g_{k}(\alpha) \leq g_{k+1}(\alpha)$ for all $\alpha \in
        \TTree$.
  \end{enumerate}
\end{lemma}
\begin{proof}
  Suppose that $f \colon \UInt \to \Real$ has a continuous ternary
  modulus.
  By Lemma~\ref{lem:ContModItself}, 
  $f$ has a continuous ternary modulus $g \colon \Nat \to \TTree
  \to \Nat$ which satisfies \ref{lem:ContItselfMonotone1}.
  Define $G \colon \Nat \to \TTree \to \Nat$ by 
  \[
    G_{k}(\alpha) \defeql \max \left\{ g_{i}(\alpha) \mid i \leq
  k\right\},
  \]
  which clearly satisfies \ref{lem:ContItselfMonotone2}.
  It is also easy to see that $G$ is a modulus of $f$. 
  To see that $G$ is a modulus of itself, let $k \in \Nat$ and
  $\alpha, \beta \in \TTree$, and suppose that 
    $
    \overline{\alpha}G_{k}(\alpha) = \overline{\beta}G_{k}(\alpha).
    $
  Then, $\overline{\alpha}g_{i}(\alpha) =
  \overline{\beta}g_{i}(\alpha)$ for all $i \leq k$. Since
  $g_{i}\; (i \leq k)$ is a modulus of itself, we have
  $g_{i}(\alpha) = g_{i}(\beta)$ for all $i \leq k$. Thus
  $G_{k}(\alpha) = G_{k}(\beta)$.
\end{proof}

\begin{lemma}
  \label{lem:ContModToCode}
  If $f \colon \UInt \to \Real$ has a continuous modulus, then
  $f$ is induced by a code of a continuous function.
\end{lemma}
\begin{proof}
  Suppose that $f \colon \UInt \to \Real$ has a continuous modulus.
  By Proposition~\ref{prop:EquivModTernMod} and Lemma~\ref{lem:ContItselfMonotone}, we may assume that 
  $f$ has a continuous ternary modulus $g \colon \Nat \to \TTree \to
  \Nat$ which satisfies \ref{lem:ContItselfMonotone1} and
  \ref{lem:ContItselfMonotone2} of Lemma \ref{lem:ContItselfMonotone}.

  For each $s \in \Ter^{*}$, define $k_{s} \in \Nat + \left\{ \bot
  \right\}$ by
  \[
    k_{s} 
    \defeql
    \begin{cases}
     \bot  & \text{if $\forall k \leq \lh{s} \left( g_{k}(\widebreve{s})
       > \lh{s}\right)$}, \\
      \text{the largest $k \leq \lh{s}$ such that $
      g_{k}(\widebreve{s}) \leq \lh{s}$} & \text{otherwise}.
    \end{cases}
  \]
  Define a function $\varphi \colon \Ter^{*} \to \Nat$ by
  \begin{equation}
    \label{def:ContModToCode}
    \varphi(s)
    \defeql 
    \begin{cases}
      0 & \text{if $k_{s} = \bot$},\\
      1 + \left(
      f(\Phi(\widebreve{\overline{s}g_{k_{s}}(\widebreve{s})}))_{k_{s}} 
      - 7 \cdot 2^{-k_{s}}, 
      f(\Phi(\widebreve{\overline{s}g_{k_{s}}(\widebreve{s})}))_{k_{s}} 
      + 7 \cdot 2^{-k_{s}} \right)
      & \text{otherwise}.
    \end{cases}
  \end{equation}
  We show that $\varphi$ is a code. 
  The property \ref{def:CodeCont1} is obvious.
  For \ref{def:CodeCont2}, let $k \in \Nat$  
  and $\alpha \in \TTree$. Since $g_{k+4}$ is continuous, there exists
  $n \in \Nat$ such that $g_{k+4}(\widebreve{\overline{\alpha}n}) \leq n$ and
  $k + 4 \leq n$. Then $k_{\overline{\alpha}n} \neq \bot$ 
  and $k + 4 \leq k_{\overline{\alpha}n}$. Thus
  $\varphi(\overline{\alpha}n) \neq 0$
  and $\lh{\varphi(\overline{\alpha}n) \dotminus 1}
  \leq 14 \cdot 2^{-k_{\overline{\alpha}n}} \leq 14 \cdot 2^{-(k + 4)} < 2^{-k}$.

  For \ref{def:CodeCont3}, let $s \in \Ter^{*}$ and
  $i \in \Ter$, and suppose that $\varphi(s) \neq 0$.
  Put $t = s*\seq{i}$. Since $g_{k_{s}}(\widebreve{s}) \leq
  \lh{s}$ and $g_{k_{s}}$ is a modulus of itself, we have
  $g_{k_s}(\widebreve{s}) = g_{k_s}(\widebreve{t}) \leq \lh{t}$.
  Hence $\varphi(t) \neq 0$ and $k_{s} \leq k_{t}$.
  We distinguish two cases:
  \smallskip

  \noindent\emph{Case $k_{s} = k_{t}$}:
  Then, $g_{k_{s}}(\widebreve{s}) = g_{k_{t}}(\widebreve{t})$, so
  $\varphi(s) \dotminus 1 = \varphi(t) \dotminus 1$.

  \smallskip
  \noindent\emph{Case $k_{s} < k_{t}$}: Then, 
   $g_{k_{s}}(\widebreve{s}) = g_{k_{s}}(\widebreve{t}) 
   \leq g_{k_{t}}(\widebreve{t})$ by the monotonicity of $g$ on the
   first argument. Thus
      $
      \overline{s}g_{k_{s}}(\widebreve{s}) \preccurlyeq 
      \overline{t}g_{k_{t}}(\widebreve{t}),
      $
   so by Lemma \ref{lem:RepresentationApprox}, we have
   \begin{align*}
     \lh{\Phi(\widebreve{s}) -
     \Phi(\widebreve{\overline{t}g_{k_{t}}(\widebreve{t})}) }
     &\leq 
     \lh{\Phi(\widebreve{s}) -
     \Phi(\widebreve{\overline{s}g_{k_{s}}(\widebreve{s})})}
     + 
     \lh{\Phi(\widebreve{\overline{s}g_{k_{s}}(\widebreve{s})})
     - \Phi(\widebreve{\overline{t}g_{k_{t}}(\widebreve{t})})} \\
     &\leq 
     2^{-(g_{k_{s}}(\widebreve{s}) + 1)}
     + 
     2^{-(g_{k_{s}}(\widebreve{s}) + 1)}
     = 
     2^{-g_{k_{s}}(\widebreve{s})}.
   \end{align*}
   Since $g$ is a modulus of $f$, 
   \begin{align*}
     &\lh{f(\Phi(\widebreve{\overline{s}g_{k_{s}}(\widebreve{s})}))_{k_{s}}
     -
     f(\Phi(\widebreve{\overline{t}g_{k_{t}}(\widebreve{t})}))_{k_{t}}}
     \\
     &\leq 
     \lh{f(\Phi(\widebreve{\overline{s}g_{k_{s}}(\widebreve{s})}))_{k_{s}}
     - f(\Phi(\widebreve{\overline{s}g_{k_{s}}(\widebreve{s})}))}
     +
     \lh{f(\Phi(\widebreve{\overline{s}g_{k_{s}}(\widebreve{s})}))
     - f(\Phi(\widebreve{s}))} \\
     &\quad\qquad\qquad +
     \lh{f(\Phi(\widebreve{s}))
     - f(\Phi(\widebreve{\overline{t}g_{k_{t}}(\widebreve{t})}))}
     +
     \lh{f(\Phi(\widebreve{\overline{t}g_{k_{t}}(\widebreve{t})}))
     -
     f(\Phi(\widebreve{\overline{t}g_{k_{t}}(\widebreve{t})}))_{k_{t}}}
     \\
     &\leq 2^{-k_{s}} + 2^{-k_{s}} + 2^{-k_{s}} + 2^{-k_{t}} \\
     &\leq  3 \cdot 2^{-k_{s}} + 2^{-(k_{s} + 1)} 
     = 7 \cdot
     2^{-(k_{s} + 1)},
   \end{align*}
   where the last $\leq$ follows from $k_{s} < k_{t}$.
   Then
   \[
     \begin{aligned}
       f(\Phi(\widebreve{\overline{s}g_{k_s}(\widebreve{s})}))_{k_s} - 7 \cdot 2^{-k_{s}} 
       &\leq 
       f(\Phi(\widebreve{\overline{s}g_{k_t}(\widebreve{t})}))_{k_t} - 7 \cdot
       2^{-(k_{s} + 1)}  \\
       &\leq
       f(\Phi(\widebreve{\overline{s}g_{k_t}(\widebreve{t})}))_{k_t} - 7 \cdot
       2^{-k_{t}},
     \end{aligned}
   \]
   and similarly
   $f(\Phi(\widebreve{\overline{s}g_{k_t}(\widebreve{t})}))_{k_t} + 7 \cdot
   2^{-k_{t}} \leq
   f(\Phi(\widebreve{\overline{s}g_{k_s}(\widebreve{s})}))_{k_s} + 7 \cdot 2^{-k_{s}}$.
   Hence $\varphi(t) \dotminus 1 \sqsubseteq
   \varphi(s) \dotminus 1$.

   For \ref{def:CodeCont4}, let $s, t \in
   \Ter^{*}$ such that $\varphi(s) \neq 0$, $\varphi(t) \neq 0$,
   and $\RInt_s \approx \RInt_t$. 
   Then, there exists $x \in \RInt_s \cap \RInt_t$ such that 
   $\lh{x - \Phi(\widebreve{s})} \leq 2^{-(\lh{s} + 1)}$
   and 
   $\lh{x - \Phi(\widebreve{t})} \leq 2^{-(\lh{t} + 1)}$.
   Then
   $\lh{x - \Phi(\widebreve{s})} \leq 2^{-g_{k_{s}}(\widebreve{s})}$
   and 
   $\lh{x - \Phi(\widebreve{t})} \leq
   2^{-g_{k_{t}}(\widebreve{t})}$ by the definitions of $k_{s}$ and
   $k_{t}$.
   Since $g$ is a modulus of $f$, we have
   $\lh{f(x) - f(\Phi(\widebreve{s}))} \leq 2^{-k_{s}}$
   and 
   $\lh{f(x) - f(\Phi(\widebreve{t}))} \leq 2^{-k_{t}}$.
   Thus
   \[
    \begin{aligned}
   &\lh{f(\Phi(\widebreve{\overline{s}g_{k_{s}}(\widebreve{s})}))_{k_{s}}
   -
   f(\Phi(\widebreve{\overline{t}g_{k_{t}}(\widebreve{t})}))_{k_{t}}}
   \\
   &\leq
   \lh{f(\Phi(\widebreve{\overline{s}g_{k_{s}}(\widebreve{s})}))_{k_{s}}
   -
   f(\Phi(\widebreve{\overline{s}g_{k_{s}}(\widebreve{s})}))
   }
   +
   \lh{f(\Phi(\widebreve{\overline{s}g_{k_{s}}(\widebreve{s})}))
   - 
   f(\Phi(\widebreve{s}))}  \\
   &\qquad
   +
   \lh{f(\Phi(\widebreve{s})) - f(x)}
   +
   \lh{f(x) - f(\Phi(\widebreve{t}))}\\
   &\qquad\qquad
   +
   \lh{f(\Phi(\widebreve{t}))
   -
   f(\Phi(\widebreve{\overline{t}g_{k_{t}}(\widebreve{t})}))}
   + 
   \lh{f(\Phi(\widebreve{\overline{t}g_{k_{t}}(\widebreve{t})}))
   -
   f(\Phi(\widebreve{\overline{t}g_{k_{t}}(\widebreve{t})}))_{k_{t}}
   } \\
   &\leq 2^{-k_{s}} + 2^{-k_{s}} + 2^{-k_{s}}  + 2^{-k_{t}} +
   2^{-k_{t}} + 2^{-k_{t}} \\
   &= 3 \cdot 2^{-k_{s}} + 3 \cdot 2^{-k_{t}}.
     \end{aligned}
   \]
   Hence $\varphi(s) \dotminus 1 \approx \varphi(t) \dotminus 1$.
   Therefore $\varphi$ is a code.

   Next, we show that $\varphi$ induces $f$. To this end, it
   suffices to show that 
   \[
     f^{\varphi}_{T}(\alpha) \simeq G(f(\Phi(\alpha)))
   \]
   for all $\alpha \in \TTree$, i.e., 
   \begin{equation}
     \label{eq:lem:ContModToCode}
     \varphi(\overline{\alpha}h_{k}(\alpha)) \dotminus 1
     \approx 
     (f(\Phi(\alpha))_{k+1} - 2^{-(k+1)},f(\Phi(\alpha))_{k+1} + 2^{-(k+1)} )
   \end{equation}
   for all $\alpha \in \TTree$ and $k \in \Nat$, where
   $h_{k}(\alpha)$ is given by \eqref{eq:hn}.
   Fix $\alpha \in \TTree$ and $k \in \Nat$, and 
   put $s = \overline{\alpha}h_{k}(\alpha)$. By definition, we have
   \[
     \varphi(s) \dotminus 1
     =
     (f(\Phi(\widebreve{\overline{\alpha}g_{k_{s}}(\widebreve{s})}))_{k_{s}}
     - 7 \cdot 2^{-k_s},
     f(\Phi(\widebreve{\overline{\alpha}g_{k_{s}}(\widebreve{s})}))_{k_{s}}
     + 7 \cdot 2^{-k_s}).
   \]
   We have
   $\lh{\Phi(\alpha) - \Phi(\widebreve{s})}
   \leq 2^{-(\lh{s}+1)} \leq 2^{-g_{k_{s}}(\widebreve{s})}$
   and 
   $\lh{\Phi(\widebreve{s}) -
   \Phi(\widebreve{\overline{\alpha}g_{k_{s}}(\widebreve{s})})}
   \leq 2^{-g_{k_{s}}(\widebreve{s})}$ by Lemma
   \ref{lem:RepresentationApprox}. Since $g$ is a modulus of
    $f$, 
   \[
     \begin{aligned}
       &\lh{f(\Phi(\alpha))_{k+1} -
       f(\Phi(\widebreve{\overline{\alpha}g_{k_{s}}(\widebreve{s})}))_{k_{s}}}\\
       &\leq
       \lh{f(\Phi(\alpha))_{k+1} -
       f(\Phi(\alpha))} 
       + 
       \lh{f(\Phi(\alpha)) - f(\Phi(\widebreve{s}))} \\
       & \qquad \qquad +
       \lh{f(\Phi(\widebreve{s}))
       - f(\Phi(\widebreve{\overline{\alpha}g_{k_{s}}(\widebreve{s})}))}
       +
       \lh{f(\Phi(\widebreve{\overline{\alpha}g_{k_{s}}(\widebreve{s})}))
       -
         f(\Phi(\widebreve{\overline{\alpha}g_{k_{s}}(\widebreve{s})}))_{k_{s}}}
         \\
         &\leq 2^{-(k+1)} + 2^{-k_{s}} + 2^{-k_{s}} + 2^{-k_{s}} \\
         &= 2^{-(k+1)} + 3 \cdot 2^{-k_{s}},
       \end{aligned}
   \]
   from which \eqref{eq:lem:ContModToCode} follows.
\end{proof}

In summary, we have the following equivalence.
\begin{theorem} \label{thm:EquivLoebContWithContMod}
  A function $f \colon \UInt \to \Real$ has a continuous modulus if
  and only if $f$ is induced by a code of a continuous function.
\end{theorem}
\begin{proof}
  By Lemma \ref{lem:CodeToContMod} and Lemma \ref{lem:ContModToCode}.
\end{proof}

Next, we characterise uniformly continuous functions from $\UInt$
to $\Real$ in terms of uniformly continuous codes.
\begin{definition}[{Loeb~\cite[Definition 3.2]{Loeb05}}]
  \label{def:CodeUCont}
  A code $\varphi \Ter^{*} \to \Nat$ is said to be \emph{uniformly
  continuous} if
  \begin{equation}
    \label{eq:CodeUifCont}
    \forall k \in \Nat \exists n \in \Nat \forall \alpha \in \TTree 
      \left( \varphi(\overline{\alpha}n) \neq 0 \wedge
      \lh{\varphi(\overline{\alpha}n) \dotminus 1} \leq 2^{-k} \right).
  \end{equation}
\end{definition}
\begin{lemma}
  \label{lem:CodeToUCont}
  If $f \colon \UInt \to \Real$ is induced by a uniformly
  continuous code, then $f$ is uniformly continuous.
\end{lemma}
\begin{proof}
  Let $\varphi \colon \Ter^{*} \to \Nat$ be a uniformly continuous
  code. By Lemma \ref{lem:UCMod}, it suffices to show that the ternary
  modulus $g$ of $f_{\varphi} \colon \UInt \to \Real$ defined by \eqref{eq:CodeTernaryMod} is uniformly
  continuous. Since $\rho \colon \TTree \to \TTree$ is uniformly
  continuous, it suffices to show that the function $h_{k} \colon \TTree
  \to \Nat$ defined by \eqref{eq:hn} is uniformly continuous for each
  $k \in \Nat$. But this clearly follows from the uniform continuity of
  $\varphi$.
\end{proof}

The following is analogous to 
Lemma \ref{lem:ContItselfMonotone}.
\begin{lemma}
  \label{lem:UContItselfMonotone}
  If $f \colon \UInt \to \Real$ is uniformly continuous,
  then $f$ has a uniformly continuous ternary modulus $g \colon \Nat \to \TTree
  \to \Nat$ such that for each $k \in \Nat$, 
  \begin{enumerate}
    \item\label{lem:UContItselfMonotone1} $g_{k}$ is a continuous
      modulus of itself,

    \item\label{lem:UContItselfMonotone2} $g_{k}(\alpha) \leq
      g_{k+1}(\alpha)$ for all $\alpha \in \TTree$.
  \end{enumerate}
\end{lemma}
\begin{proof}
  Let $\omega \colon \Nat \to \Nat$ be a modulus of uniform
  continuity of $f$. For each $k \in \Nat$, defined $g_{k} \colon
  \TTree \to \Nat$ by 
  \[
    g_{k}(\alpha) \defeql \max \left\{ \omega(i) \mid i \leq
    k\right\}.
  \]
  Then, $g_{k}$ is trivially uniformly continuous which is a modulus of
  $f$ and of itself and is monotone on the first argument.
\end{proof}

\begin{lemma}
  \label{lem:UContToCode}
  If $f \colon \UInt\to \Real$ is uniformly continuous, then  $f$ is
  induced by a uniformly continuous code.
\end{lemma}
\begin{proof}
  Suppose that $f$ is uniformly continuous. By Lemma
  \ref{lem:UContItselfMonotone}, $f$ has a uniformly continuous
  ternary modulus $g \colon \Nat \to \TTree \to \Nat$ which satisfies
  \ref{lem:UContItselfMonotone1} and~\ref{lem:UContItselfMonotone2} of 
  Lemma \ref{lem:UContItselfMonotone}.
  By the proof of Lemma \ref{lem:ContModToCode}, it suffices to show that the code
  $\varphi$ defined by \eqref{def:ContModToCode} is uniformly
  continuous.
  Fix $k \in \Nat$. Since $g_{k+4}$ is uniformly continuous, there
  exists
  $n \in \Nat$ such that 
  \[
    \forall \alpha,\beta \in \TTree \left( \overline{\alpha}n =
    \overline{\beta}n \imp g_{k+4}(\alpha)  = g_{k+4}(\beta)\right).
  \]
  Put $M = \max \left\{ \max \left\{ g_{k+4}(\widebreve{s}) \mid s \in
  \Ter^{n}\right\}, k+4, n \right\}$.
  Let $\alpha \in \TTree$. Then
  $g_{k+4}(\widebreve{\overline{\alpha}M})
  = g_{k+4}(\widebreve{\overline{\alpha}n})
  \leq M$.
  Since $k+4 \leq M$, we have $k+4 \leq k_{\overline{\alpha}M}$.
  Thus $14 \cdot 2^{-k_{\overline{\alpha}M}} \leq 14 \cdot 2^{-(k+4)}
  < 2^{-k}$. Hence $\varphi(\overline{\alpha}M) \neq 0$
  and $\lh{\varphi(\overline{\alpha}M))\dotminus 1} \leq 2^{-k}$.
\end{proof}

In summary, we have the following equivalence.
\begin{theorem} \label{thm:EquivLoebContUCont}
  A function $f \colon \UInt \to \Real$
  is uniformly continuous if and only if $f$ is induced by
  a uniformly continuous code.
\end{theorem}
\begin{proof}
  By Lemma \ref{lem:CodeToUCont} and Lemma \ref{lem:UContToCode}.
\end{proof}

\subsection*{Acknowledgements}
Part of this work was carried out in October 2019 at the
Zukunftskolleg of the University of Konstanz, which was hosting the
first author as a visiting fellow. The authors thank the institute for
their support and hospitality. The first author was supported by JSPS KAKENHI
Grant Numbers JP18K13450, JP19J01239, and JP20K14354. 
The second author was supported by JSPS KAKENHI Grant Number JP20K14352.
%\bibliographystyle{abbrv}
%\bibliography{references}

\end{document}